\documentclass[12pt,leqno]{article}
\usepackage[utf8]{inputenc}
\usepackage{mathrsfs}
\usepackage{indentfirst,amsmath,amsthm,amsfonts,amsthm,amssymb}
\usepackage{hyperref}
\usepackage{color}

\def\P{\mathbb{P}_\Omega}

\def\wb{\overline{w}}
\def\tu{\widetilde{u}}
\def\la{\lambda}

\def\wb{\overline{w}}

\newcommand\nlp[2]{\left\|#2\right\|_{L^{#1}(\Omega)}}
\newcommand\nlpc[2]{\left\|#2\right\|_{L^{#1}(B_R)}}
\newcommand\nlpr[2]{\left\|#2\right\|_{L^{#1}(\R^n)}}

\newcommand\nlpw[2]{\left\|#2\right\|_{L^{{#1},\infty}(\Omega)}}
\def\tu{\widetilde{u}}
\def\ep{\varepsilon}
\def\al{\alpha}
\def\om{\omega}

\def\Om{\Omega}

\def\loc{_{{\rm loc}}}

\def\wu{\widetilde{u}}
\def\L2s{L^2_\sigma (\Omega)}
\def\Lps{L^p_\sigma (\Omega)}
\def\Lqs{L^q_\sigma (\Omega)}
\def\LpO{{L^p(\Omega)}}

\newtheorem{theorem}{Theorem}[section]
\newtheorem{lemma}[theorem]{Lemma}%[section]
\newtheorem{definition}[theorem]{Definition}%[section]
\newtheorem{proposition}[theorem]{Proposition}%[section]
\newtheorem{corollary}[theorem]{Corollary}%[section]

\numberwithin{equation}{section} 

\theoremstyle{remark}
\newtheorem{remark}[theorem]{Remark}

\newcommand{\real}{\mathbb{R}}

\newcommand{\be}{\begin{equation}}
\newcommand{\ee}{\end{equation}}
\renewcommand{\leq}{\leqslant}
\renewcommand{\geq}{\geqslant}

\def\R{\mathbb{R}}
\def\ppom{\mathbb{P}_{\Omega}}

\DeclareMathOperator\curl{curl}
\DeclareMathOperator\dive{div}
\DeclareMathOperator\supp{supp}
\DeclareMathOperator\mes{mes}

\def\vb{\overline{v}}

\makeatletter
\def\blfootnote{\gdef\@thefnmark{}\@footnotetext}
\makeatother

\title{\bf Asymptotics of solutions to the Navier-Stokes system in exterior domains}

\author{Drago\c s Iftimie%\footnote{Corresponding author}
, Grzegorz Karch \& Christophe  Lacave}
\def\adrese{
\begin{description}
\item[D. Iftimie:] Universit\'e de Lyon, CNRS, Universit\'e Lyon 1, Institut Camille Jordan, 43 bd. du 11 novembre, Villeurbanne Cedex F-69622, France.\\
Email: \texttt{iftimie@math.univ-lyon1.fr}\\
Web page: \texttt{http://math.univ-lyon1.fr/\~{}iftimie} 
\item[G. Karch:] Instytut Matematyczny, Uniwersytet Wroc{\l}awski,
 pl. Grunwaldzki 2/4, 50-384 Wroc{\l}aw, Poland.\\
Email: \texttt{grzegorz.karch@math.uni.wroc.pl}\\
Web page: \texttt{http://www.math.uni.wroc.pl/\~{}karch}
\item[C. Lacave:] Universit\'e Paris-Diderot (Paris 7), Institut de Math\'ematiques de Jussieu - Paris Rive Gauche, UMR 7586 - CNRS, B\^atiment Sophie Germain, Case 7012, 75205 PARIS Cedex 13, France.\\
Email: \texttt{lacave@math.jussieu.fr}\\
Web page: \texttt{http://people.math.jussieu.fr/\~{}lacave/} 
\end{description}
}
\date{}
\begin{document} 
\maketitle
\begin{abstract}
We consider the incompressible Navier-Stokes equations with the Dirichlet boundary condition in an exterior domain of $\mathbb{R}^n$ with  $n\geq2$. 
We compare the long-time behaviour of  solutions to this initial-boundary value problem with the long-time behaviour of  solutions 
of the analogous Cauchy problem 
in the whole space $\mathbb{R}^n$. 
We find that the long-time asymptotics of solutions to both problems
 coincide either  in the case of small initial data in the weak $L^{n}$-space or
  for a certain class of large initial data.
\end{abstract}

\section{Introduction}

\blfootnote{Mathematics Subject Classification: 35Q30, 76D05, 35Q35.}

\subsection*{Large time behaviour of solutions.}
It is well-known that the large time behaviour of  solutions to the initial-value problem for 
the Navier-Stokes equations considered either in the whole space $\R^n$ with $n\geq 2$ or in an exterior domain
depends on the integrability properties of the  initial conditions.
 In the finite energy case, that is when the initial  velocity is square integrable, the
$L^2$-norm of  the corresponding 
solution tends to zero as 
 time goes to infinity, see {\it e.g.}  \cite{kato_1984,MR775190,wiegner,kajikiya-miyakawa} 
for the problem in the whole space $\R^n$ and \cite{MR1158939} for 
analogous results  in exterior domains.  
In such a  case, the nonlinear effects are negligible for large values of time and the asymptotics of  the solutions is determined by the corresponding Stokes semigroup, {\it cf.~e.g.}~\cite[Thm.~2]{kajikiya-miyakawa} .

On the other hand, when the initial velocity is not square integrable, a solution of the initial value problem for the
Navier-Stokes equations  in  $\R^n$  
 can be constructed in a so-called  scaling invariant space 
({\it e.g.}~in a homogeneous Besov space or
in a weak $L^n$-space) under a suitable smallness assumption on  the initial data, see the
review article by Cannone \cite{cannone} and the book of
Lemari\'e-Rieusset 
\cite{MR1938147} for more details. 
In particular, 
if the initial velocity is small and  homogeneous of degree $-1$, the corresponding solution is self-similar and such self-similar solutions describe
 the large time behaviour of a large class of solutions of the Navier-Stokes system in $\R^n$, see {\it e.g.} \cite{MR1639283,MR1689406}
and \cite[Ch.~23]{MR1938147}. 
Here, the asymptotic of  solutions  is no longer determined by the Stokes semi-group due to the fact that the viscosity term and the bilinear term 
are in exact  balance in the sense that  none of them dominates
the large time behaviour.  

Finally, it should be noted that, for certain initial data, the
large time behaviour of the corresponding solutions can be much more involved.
Indeed, the authors of  \cite{cazenave_chaotic_2005}  noticed a chaotic behaviour of some solutions of the Navier-Stokes system, namely, the sequence $\{u(t_n,x)\}$  may exhibit different  asymptotic properties
  as $t_n\to\infty$, depending on a choice of the  sequence $\{t_n\}$.

\subsection*{Lamb-Oseen vortex in two dimensions.}
Now, we focus for a moment on the self-similar large time behaviour of solutions of the two dimensional Navier-Stokes system.
As an important and physically relevant case, let us recall that 
a velocity in $\R^2$ corresponding  to an integrable 
vorticity  is not square integrable in general. This is an immediate consequence of the Biot-Savart law (see {\it e.g.}~\cite{MR1017289}). 
The large time asymptotic of solutions of the Navier-Stokes system in $\R^2$ supplemented with such initial 
conditions
is well-understood. To recall this result, 
 let us introduce the Lamb-Oseen vortex
\begin{equation*}
\Theta(t,x)= \frac{x^\perp}{2\pi|x|^2}\Bigl(1-e^{-\frac{|x|^2}{4t}}\Bigr)
\qquad \text{with}\quad  x^\perp =(x_2,-x_1),
\end{equation*}
which is an explicit self-similar solution of the system in $\R^2$ corresponding to  the initial velocity $\Theta_0(x)=\frac{x^\perp}{2\pi|x|^2}$ and the initial vorticity $\curl\Theta_0=\delta_0$ (the Dirac mass). 
The Lamb-Oseen vortex $\Theta$
is known to characterize the large time behaviour of the solutions of the Navier-Stokes system in $\R^2$ supplemented with the initial datum $u_0$ satisfying 
$\om_0=\curl u_0\in L^1(\R^2)$ in the following sense 
\begin{equation*}
  \lim_{t\to\infty}t^{\frac12-\frac1p}\|u(t)-m\Theta(t)\|_{L^p}=0\qquad\text{for all}\quad  p\in (2,\infty],
\end{equation*}
where $m=\int_{\R^2}\omega_0\;dx$. This result  is due to 
Giga \& Kambe \cite{GK88} if $\|\om_0\|_{L^1}$ is small, to Carpio
\cite{Ca94} in the case when $\int_{\R^2}\omega_0\;dx  $ is small, and 
to Gallay \& Wayne
\cite{GW05} in the general case (together with the higher order terms in the asymptotic expansion of solutions).

Such general asymptotic results are not known for problems in two dimensional exterior domains.
  It is  not even clear whether  the hypothesis on the integrability of  the
 initial 
vorticity is   relevant. Here,
no  $L^1$-bound for the vorticity is known  because of the absence of any reasonable boundary condition for the vorticity.
Thus, to overcome this technical problem
in our unpublished manuscript \cite{ILK11}, 
we assumed that the initial  velocity behaves for large values of $|x|$ 
like a multiple of $\frac{x^\perp}{|x|^2}$.
More precisely, we showed in \cite{ILK11}
that if $u_0(x)=w_0(x)+\al\frac{x^\perp}{|x|^2}$ with $w_0\in L^2(\Om)$ and $\alpha\in\R$, then there exists $\al_0=\al_0 (w_0,\Omega)>0$ such that  for all  $|\al|\leq \al_0$ we have that 
\begin{equation*}
  \lim_{t\to\infty}t^{\frac12-\frac1p}\|u(t)-\al\Theta (t)\|_\LpO=0
\end{equation*}
for each $p\in (2,\infty)$.
Comparing with the analogous result in the full plane case, 
we have an additional smallness condition on  the parameter $\alpha$, which is  due to the fact that neither $L^2$-estimates for the velocity nor $L^p$-estimates for the vorticity are known for the problem in an exterior domain (because the velocity is not square-integrable and because of the absence of boundary conditions for the vorticity). 
Consequently, we have proved in  \cite{ILK11} that 
 the large time asymptotics of the solutions of the problem in an exterior domain is the same as in the full plane case and is given by   the Lamb-Oseen vortex. 
This unpublished result  is now a particular case of Theorem~\ref{dim2}, below (see Remark \ref{rem:d2:old}).
Our result from  \cite{ILK11}  has been recently improved by Gallay \& Maekawa \cite{GM11}, where 
 the smallness constant $\al_0$ can be chosen independently of $w_0$ if one 
imposes the additional minor assumption that $w_0\in L^q(\Omega)$ for some $q<2$.

\subsection*{More general slowly decaying initial conditions.}

Even though the most physically interesting  case in two dimensions occurs when $u_0(x)\simeq C\frac{x^\perp}{|x|^2}$ at infinity, one could consider other behaviours of $u_0$ at infinity. 
For example, one could assume that the initial velocity  behaves at infinity 
like an arbitrary  divergence-free homogeneous vector field of degree $-1$. 
More generally, one can consider initial velocities from  
the Marcinkiewicz (weak $L^2$) space $L^{2,\infty}(\Omega)$;
however, in such a general setting and  due to
the chaotic behaviour observed in \cite{cazenave_chaotic_2005},
 we may not actually have any clear asymptotics of solutions for large values of times. 

To avoid such a difficulty, instead of looking for the 
 asymptotic behaviour of solutions to the problem in an exterior domain directly,  we will compare the
solutions of 
the Navier-Stokes system in an exterior domain   $\Omega\subset \R^n$ $(n\geq 2)$  to the solutions of the Navier-Stokes system in the whole space $\R^n$  and we show that the solutions to both problems behave, as $t\to\infty$, 
in the same way provided that   their initial data are {\it ``comparable at infinity''} (meaning that the difference is slightly better than $L^{n,\infty}$ at infinity, see the definition below). 
This approach allows us to remove the obstacle from the problem under considerations
and to reduce the study of the Navier-Stokes system in an exterior domain 
to the  study of this system in the whole space $\R^n$ (at least as far as large time asymptotics are concerned).

We discuss now what kind of initial conditions can be used in our approach.

\begin{definition}\label{def:eq}
 Let $v_0$ be a divergence free vector field defined on $\R^n$ 
and let $u_0$ be a divergence free vector field defined on an exterior domain $\Omega\subset \R^n$. Assume moreover that $u_0$ is  tangent to the boundary of $\Omega$, {\it i.e.~}$u\cdot \nu=0$ on the boundary, where $\nu$ is the normal vector to the boundary.
We say that $u_0$ and $v_0$ are {\it comparable at infinity}
 if $u_0-v_0\vert_{\Omega}$ belongs to $L^{q_0}(\Omega)$ for some $q_0\in (1,n]$.
\end{definition}
Given $v_0\in L^{n,\infty}(\R^n)$ a divergence free vector field, examples of vector fields $u_0$ comparable at infinity are $u_0=\ppom v_0\bigl|_\Om$ and also the vector field obtained from $v_0$ by the truncation procedure described in Section~\ref{sec:linear} (which in dimension two  corresponds to cuting-off the stream function). Here, $\ppom$ is the Leray projector associated to $\Om$, \textit{i.e.}~the $L^2$-orthogonal projection on the space of vector fields which are divergence free and tangent to the boundary of $\Omega$. 
Note also that if $v_0$ is the extension of $u_0$ to $\R^n$ with zero values on $\R^n\setminus\Om$, then $u_0$ and $v_0$ are comparable at infinity. Indeed, in this case, we have that  $u_0=\ppom v_0\bigl|_\Om$, 
see below  for more details.

\subsection*{Results of this work.}

Now, we briefly present the main results of this work.
Let $v$ be a solution of the Navier-Stokes equations in $\R^n$ with initial velocity $v_0$:
\begin{equation*}
 \partial_t v-\Delta v+v\cdot\nabla v=-\nabla p \quad\text{in }\R^n, \qquad v\bigl|_{t=0}=v_0,
\end{equation*}
and let  $u$ be a solution of the Navier-Stokes equations in the exterior domain $\Omega\subset \R^n$  $(n\geq 2)$ with initial velocity $u_0$ and with the homogeneous Dirichlet boundary conditions:
\begin{equation}\label{nsext}
\partial_t u-\Delta u+u\cdot\nabla u=-\nabla q \quad\text{in }\Omega, \qquad u\bigl|_{t=0}=u_0,\qquad u\bigl|_{\partial\Om}=0.    
\end{equation}
The main result of this paper, formulated in  Theorem  \ref{corwithsmallu} below,  states  that if $u_0$ and $v_0$ are comparable at infinity and small in the norm of the space $L^{n,\infty}$, then $u$ and $v$ have the same large time behaviour in the following sense 
\begin{equation}\label{intro:asymp}
\lim_{t\to\infty}t^{\frac12-\frac{n}{2p}}\nlp p{u(t)-v(t)}=0
\end{equation}
for all $n<p<\infty$.
\begin{remark}
The decay rate in \eqref{intro:asymp} corresponds to the optimal decay  of  solutions measured in $L^p$-norms.
Indeed, if $v_0\in L^{n,\infty}(\R^n)$ is small and homogeneous of degree $-1$, then the corresponding solution is self-similar: $v(x,t)=t^{-1/2}v(x t^{-1/2},1)$ for all $x\in \R^n$ and $t>0$.
Hence, changing variables, we obtain $t^{\frac12-\frac{n}{2p}}\nlp p{v(t)}=\nlp p{v(1)}$ for all $t>0$.
\end{remark}

Next, we can extend this large time asymptotics result to a class of large initial data. In  dimension $n\geq3$, we can replace the smallness
of $u_0,v_0$  in the Marcinkiewicz space $L^{n,\infty}(\Omega)$ 
by the smallness of both quantities 
$$
\limsup\limits_{\la\to0}\la\mes\{x\in\Omega\,:\, |u_0(x)|>\la\}^{\frac1n}\qquad \text{and}\qquad  
\limsup\limits_{\la\to0}\la\mes\{x\in\R^n\,:\, |v_0(x)|>\la\}^{\frac1n}
$$ 
in order to show in  Theorem \ref{largedata} below
that 
the limit \eqref{intro:asymp} is again valid.
However, with this improved smallness assumption, the uniqueness of the solutions is no longer ensured and our large time asymptotics result holds true for some weak solutions, only. 

In  dimension two, the statement of our large data asymptotic result is more involved and we refer the reader to Theorem \ref{dim2} for its precise formulation. 
Here, we only point out the following special case. 
The limit in \eqref{intro:asymp}   holds true 
for the problem in an exterior two dimensional domain under the following hypothesis. The initial velocity $u_0$ for the exterior domain problem can be decomposed in  the form $u_0=\tu_0+w_0$, where $\tu_0\in L^2(\Om)$ is divergence free and tangent to the boundary (and arbitrarily large) and $w_0\in L^{2,\infty}(\Om)\cap B^{-1/2}_{4,4}(\Om) $ verifies the smallness assumptions  $\nlpw2{w_0}<\ep_1(\Om)$ and $\|w_0\|_{  B^{-1/2}_{4,4}(\Om)}<\ep_2(\tu_0)$ for some small positive constants $\ep_1$ and $\ep_2$. The initial velocity $v_0$ for the problem in $\R^2$ can be decomposed in a similar manner $v_0=\widetilde U_0+W_0$ with similar conditions on $\widetilde U_0$ and $W_0$. Moreover, $w_0$ is assumed to be comparable at infinity to $W_0$.

We mention now two applications of our results. The first one is that if the initial velocity for the exterior domain problem is small in $L^{n,\infty}(\Omega)$ and comparable at infinity to a velocity field homogeneous of degree -1, then the large time behaviour of the solution is a self-similar behaviour. The second one is that there exists an initial data for the exterior domain problem such that the corresponding solution exhibits a chaotic behaviour at infinity. Indeed, it suffices to consider $v_0$ the example of initial data in $\R^n$ exhibited in \cite{cazenave_chaotic_2005} and set $u_0=\ppom v_0\bigl|_\Om$ as initial velocity for the exterior domain problem.

The remainder of this paper is organized in the following manner. In the next section, we introduce the notation, we recall the decay estimates for
the Stokes semigroup, and we show some preliminary technical lemmas concerning initial data comparable at infinity.
 In Section~\ref{sect:withsmallu}, we prove our asymptotic result in the case of small initial data. The case of large data in dimension $n\geq3$ is considered in Section \ref{sectdim3}. 
Section \ref{sectdim2} deals with 
the  asymptotic  behaviour of solutions with
large initial conditions in dimension two.

 \section{Preliminaries} \label{sec:linear}

\subsection*{Notation}
Let $\Om\subset \R^n$ $(n\geq 2)$ be an  exterior domain with a smooth boundary $\Gamma$,  and choose $R>0$ such that $\R^n\setminus\Om\subset B_{R/2}$. 
Here, we set $B_{R/2}=B(0,R/2)$ for the ball of radius $R/2$ centered at the origin. 
We denote by $\P$ the Leray projector in $\Om$, \textit{i.e.} the $L^2$ orthogonal projection from $L^2(\Om)$ on the subspace of divergence free vector fields which are tangent to the boundary $\Gamma$. 
It is well-known that  $\P$ extends to a bounded operator on every $L^p(\Om)$, $1<p<\infty$. We denote by $L^p_\sigma(\Omega)$  the closure of the set of smooth, divergence-free, and compactly supported  vector fields  $C^{\infty}_c(\Om)$ with respect to the usual $L^p$-norm.  The space  $L^p_\sigma(\Omega)$ can also be viewed as the image of   $L^p(\Omega)$ by $\ppom$. 
In a similar way, we write that $u\in L^{p,\infty}_\sigma(\Omega)$ when $u\in L^{p,\infty}(\Omega)^n$, 
$\dive u =0$ in $\Om$, and $u \cdot \nu=0$
on the boundary $\Gamma$, where 
$ L^{p,\infty}(\Omega)$ denotes the usual weak $L^p$-space (the Marcinkiewicz space) and 
$\nu$ 
is the normal vector to the boundary $\partial\Omega$. 
In this work, we use systematically a fixed cut-off function $f\in C^\infty(\R^n, \R_{+})$ 
such that $f(x)=0$ for $|x|<R/2$ and $f(x)=1$ for $|x|>2R/3$. 

\subsection*{Stokes semigroup}
The stationary Stokes operator 
$A=-\P\Delta$ generates an analytic semigroup of linear operators 
 $\{e^{-tA}\}_{t\geq 0}$ 
 on $\Lps$ for each $1<p<\infty$, {\it cf.}~\cite{MR635201}.
If $v_0$ is divergence free and tangent to the boundary of $\Omega$, then $v(t)=e^{-tA}v_0$ is the solution of the following linear boundary value problem
\begin{align*}
&\partial_t v -\Delta v  +\nabla p =0, \qquad \dive v=0  & &\text{for}\quad t>0, \quad x\in \Om, \\
&v(t,x)=0  & &\text{for}\quad t>0, \quad x\in \Gamma,\\
&v(0,x)=v_0(x) & & \text{for} \quad x\in \Om.
\end{align*}

 The Stokes semigroup $\{e^{-tA}\}_{t\geq 0}$
satisfies the  following $L^p$-decay estimates.

\begin{proposition} \label{ellpest}
Assume that  $1 < q <\infty$.

Let  $q \leq p\leq\infty$. There exists $K_1 = K_1(\Om,p,q)>0$ such that for every $v_0 \in \Lqs$
\begin{equation}\label{S1}
\|e^{-tA}v_0\|_{L^p(\Omega)} \leq K_1 t^{\frac{n}{2p} - \frac{n}{2q}}\|v_0\|_{L^q(\Omega)}
\qquad \text{for all}\quad  t>0.
\end{equation} 
If,  in addition, we assume that $q<p\leq\infty$, then   for every $v_0\in L^{q,\infty}_\sigma(\Omega)$ we also have that 
\begin{equation}\label{S1bis}
 \|e^{-tA}v_0\|_{L^p(\Omega)} \leq K_1 t^{\frac{n}{2p} - \frac{n}{2q}}\|v_0\|_{L^{q,\infty}(\Omega)}
\qquad \text{for all}\quad  t>0.
\end{equation}

There exists $K_2 = K_2(\Om,q)>0$ such that 
for every $v_0\in L^{q,\infty}_\sigma(\Omega)$ we have 
the inequality
\begin{equation}\label{S1bis2}
 \|e^{-tA}v_0\|_{L^{q,\infty}(\Omega)} \leq K_2 \|v_0\|_{L^{q,\infty}(\Omega)}
\qquad \text{for all}\quad  t>0.
\end{equation}

There exists $K_3 = K_3(\Om,q)>0$ such that 
for every $v_0\in L^{q}_\sigma(\Omega)$ we have 
the inequality
\begin{equation}\label{analytic}
 \|Ae^{-tA}v_0\|_{L^{q}(\Omega)} \leq K_3 t^{-1}\|v_0\|_{L^{q}(\Omega)}
\qquad \text{for all}\quad  t>0.
\end{equation}

Finally, if $n \leq q \leq p < \infty$ then  there exists $K_4= K_4(\Om,p,q)>0$
such that  for every matrix valued function $F \in L^q(\Omega;M_{n}(\real))$ we have that
\begin{equation}\label{S3}
 \| e^{-tA}\ppom \dive F\|_{L^p(\Omega)} \leq K_4 t^{-\frac{1}{2} + \frac{n}{2p} - \frac{n}{2q}}
\|F\|_{L^q(\Omega)}
\qquad \text{for all}\quad  t>0,
\end{equation}
with the divergence operator  $\dive$ computed  along the rows of the matrix $F$.

\end{proposition}

 Estimates \eqref{S1}--\eqref{S1bis2} were proved
in \cite{DS99a,DS99,KY95,MS97}.  Relation \eqref{analytic} is a consequence of the fact that $e^{-tA}$ is an analytic semigroup, see \cite{MR635201}.

The following corollary contains a minor improvement of the decay estimate \eqref{S1}.

\begin{corollary}\label{cor:dec}
Assume that  $1 < q <\infty$ and let $v_0 \in \Lqs$. Then for every  $p\in [q,\infty)$
\begin{equation*}
\lim_{t\to\infty} t^{ \frac{n}{2q}-\frac{n}{2p}} \|e^{-tA}v_0\|_{L^p(\Omega)}=0. 
\end{equation*} 
\end{corollary}
\begin{proof}
The validity of this limit is clear when the initial datum $v_0$ is smooth and compactly supported.
To show it for all   $v_0 \in \Lqs$, it suffices to use a standard density argument combined with estimate \eqref{S1}.
\end{proof}

\subsection*{Initial data comparable at infinity}
Before  we compare the large time behaviour of the Navier-Stokes equations in exterior domains with the large time behaviour of the Navier-Stokes equations in the whole $\R^n$,
 we have to clarify the issue of the  initial data. 
More precisely, given a divergence-free initial datum  $v_0$ on $\R^n$, we would like to construct an initial datum $u_0$ for the problem in the exterior domain which is comparable at infinity with $v_0$. 
The simplest approach which consists in taking the restriction of  $v_0$ to $\Om$, 
{\it i.e.} $u_0=v_0\bigl|_{\Om}$,  does not work because in general this restriction is  not  
tangent to the boundary $\Gamma$. 
Thus,
in order to obtain a vector field which is divergence free and tangent to the boundary, the most natural way is to define $u_0$ by applying the Leray projection to the restriction of $v_0$ to $\Om$:
  \begin{equation}\label{defproj}
    u_0=\P (v_0\bigl|_{\Om}).
  \end{equation}
{\it Vice versa}, given a velocity field  $u_0$ on $\Om$ which is 
divergence free and tangent to the boundary,
we can construct a velocity field $v_0$ in $\R^n$, simply
 by extending $u_0$ with zero values outside $\Om$. 
Here, the divergence free condition is preserved across the boundary because $u_0$ is tangent to the boundary. Hence, with this choice of $v_0$
we clearly have relation \eqref{defproj}.

Unfortunately, defining $u_0$ as in \eqref{defproj} is not  practical from the point of view of  estimates,
 because $\P (v_0\bigl|_{\Om})$ does not verify the Dirichlet boundary condition. 
Instead, we  will use  a cutoff procedure that we detail now.

First, we need to prove a technical result.
\begin{lemma}\label{defpsic}
Let $1<p<\infty$ and $v\in L^p(\R^n)$ be a vector field. There exists a unique matrix valued function $\psi$ such that 
\begin{equation}\label{defpsil}
\psi\in W^{1,p}\loc(\R^n),\quad \nabla\psi\in L^p(\R^n), \quad \int_{B_R}\psi=0\quad\text{and}\quad \Delta\psi_{ij}=\partial_j v_i-\partial_i v_j
\end{equation}
for all $i,j$. Moreover, there exists a constant $C=C(p,R)$ such that
\begin{equation}\label{boundS}
\|\psi\|_{L^p(B_{R})}+\nlpr p{\nabla \psi}\leq C\nlpr pv.
\end{equation}
\end{lemma} 
\begin{proof}
We start by proving the uniqueness of $\psi$. Let $\psi_1$ and $\psi_2$ verify \eqref{defpsil}. Then $\nabla (\psi_1-\psi_2)$ is harmonic and belongs to $L^p(\R^n)$, therefore it must vanish so $\psi_1-\psi_2=\mathcal{C}$ for some constant matrix $\mathcal{C}$. But $\int_{B_R}\mathcal{C}=\int_{B_R}\psi_1-\int_{B_R}\psi_2=0$, hence $\mathcal{C}=0$.

We show now the existence of $\psi$. The operators $\partial_i\partial_j\Delta^{-1}$ are bounded on $L^p(\R^n)$ as  products of the Riesz transforms. Therefore $w_{ij}=\nabla\Delta^{-1}(\partial_j v_i-\partial_i v_j)$ is well defined and belongs to $L^p(\R^n)$. Now, if $\varphi$ is a divergence free, smooth and compactly supported vector field then we clearly have that
$$
\int_{\R^n}w_{ij}\cdot\varphi
=\int_{\R^n}\nabla\Delta^{-1}(\partial_j v_i-\partial_i v_j)\cdot\varphi
=\int_{\R^n}(v_i\partial_j-v_j\partial_i)\Delta^{-1}\dive\varphi=0.
$$
We infer from \cite[Lemma III.1.1 and Corollary II.4.1]{galdi}) that there exists $\psi_{ij}\in W^{1,p}\loc(\R^n)$ such that $w_{ij}=\nabla\psi_{ij}$. Adding a constant if necessary, we can also assume that $\int_{B_R}\psi_{ij}=0$. Finally, we remark that $\Delta\psi_{ij}=\dive\nabla\psi_{ij}=\dive w_{ij}=\dive \nabla\Delta^{-1}(\partial_j v_i-\partial_i v_j)=\Delta\Delta^{-1}(\partial_j v_i-\partial_i v_j)=\partial_j v_i-\partial_i v_j$. Therefore $\psi$ has all properties stated in \eqref{defpsil}. Moreover, the bound $\nlpr p{\nabla \psi}\leq C\nlpr p{w}\leq C\nlpr pv$ is obvious. Finally, since  $\psi$ has vanishing mean on the ball $B_R$, we can apply the Poincaré inequality to obtain that $\nlpc p{\psi}\leq C\nlpc p{\nabla \psi}\leq C\nlpr pv$. This completes the proof.
\end{proof}

It is well-known that $L^{p,\infty}$ can be defined by the real interpolation method: $L^{p,\infty}_\sigma=(L_\sigma^{p_0},L^{p_1}_\sigma)_{\theta,\infty}$, where $p_0<p<p_1$ and $1/p=(1-\theta)/p_0+\theta/p_1$. Since the application $v\mapsto \psi$ is obviously linear, the theory of the real interpolation implies that $\psi$ is well defined for $v\in L^{p,\infty}(\R^n)$ and relation \eqref{boundS} remains true in weak $L^p$-spaces:
\begin{equation}\label{boundSlorentz}
\|\psi\|_{L^{p,\infty}(B_{R})}+\nlpr {p,\infty}{\nabla \psi}\leq C\nlpr {p,\infty}v
\end{equation}
for all $v\in L^{p,\infty}(\R^n)$.

Observe next that if $v$ is also divergence free, then 
\begin{equation*} 
v=\dive\psi  
\end{equation*}
where $\psi$ is defined in Lemma \ref{defpsic} and the divergence of the matrix $\psi$ is computed along its rows. Remark also that the divergence of a  skew-symmetric matrix is a divergence free vector field.

Now,  let $f$ be a smooth cut-off function from $\R^n$ to $\R_+$ such that $f$ vanishes for $|x|<R/2$ and $f(x)=1$ for $|x|>2R/3$.
In this work, we consider initial data for the exterior domain problem obtained as the cut-off of the matrix $\psi_0$ associated to $v_0$  as in Lemma \ref{defpsic}:
  \begin{equation*}
    \vb_0=\dive(f\psi_0).
  \end{equation*}
Since $\psi_0$ is skew-symmetric, we have  $\dive \vb_0=0$. Moreover, from the localization property of the cut-off $f$, we infer that $\vb_0$ vanishes on the boundary $\Gamma$ and  $\vb_0=v_0$ for $|x|>R$. We also have the following lemma.

\begin{lemma}\label{lem:mapping}
The mapping $v_0\mapsto \vb_0$ 
 is bounded from 
$L^p_\sigma(\R^n)$ into $L^p_\sigma(\Om)$ for each  $1<p<\infty$ as well as  from 
$L^{n,\infty}_\sigma(\R^n)$ into $L^{n,\infty}_\sigma(\Om)$. 
\end{lemma}

\begin{proof}
Since $\vb_0=fv_0+\psi_0\nabla f$, where $\nabla f$ is supported in the ball $B_R$, we may
 use inequality \eqref{boundS} to obtain the estimate 
\begin{equation*}
\nlp p{\vb_0}\leq \nlpr p{v_0}+ \|\nabla f\|_{L^\infty}\nlpc p{\psi_0} \leq C\nlpr p{v_0}.
\end{equation*}
Therefore, the operator $v_0\mapsto \vb_0$ is linear and bounded from $L^p_\sigma(\R^n)$ into $L^p_\sigma(\Om)$ for each $1<p<\infty$. By  interpolation, it is also bounded from $L^{n,\infty}_\sigma(\R^n)$ into $L^{n,\infty}_\sigma(\Om)$. 
\end{proof}

Let us prove that both $\vb_0$ and $u_0$ chosen as in \eqref{defproj} are comparable at infinity to $v_0$.

\begin{lemma}\label{restrictions}
Let $v_0\in L^{n,\infty}(\R^n)$ be divergence free
and construct the matrix $\psi_0$ from $v_0$ as in Lemma \ref{defpsic}.
Then, all three vector fields
$$
v_0\bigl|_\Om, \qquad \ppom(v_0\bigl|_\Om),\qquad  \dive(f\psi_0)
\qquad \text{are comparable at infinity.}
$$
More precisely, the differences 
$$
 v_0\bigl|_\Om-\ppom(v_0\bigl|_\Om) \quad\text{and}\quad 
v_0\bigl|_\Om -\dive(f\psi_0) \quad\text{belong to}\quad  L^q(\Om)
$$
for each $q\in(1,n)$.   
\end{lemma}
\begin{proof}
To show that $v_0-\dive(f\psi_0)\in L^q(\Om)$  for each $q\in(1,n)$, we observe that 
$v_0-\dive(f\psi_0)=0$ for $|x|>R$ (due to properties of the cut-off $f$) and that 
$v_0-\dive(f\psi_0)\in L^{n,\infty}(B_R)$ (see relation \eqref{boundSlorentz}). 
Next, it suffices to use the imbedding $L^{n,\infty}(B_R) \subset L^{q}(B_R)$ for each $q\in(1,n)$.

Finally, using that the Leray projector is bounded in $L^q$, we have  that $\ppom(v_0-\dive(f\psi_0))\in L^q(\Om)$  for all $q\in(1,n)$. But $\dive(f\psi_0)$ is divergence free and vanishes on the boundary, so $\ppom\dive(f\psi_0)=\dive(f\psi_0)$. We infer that $\dive(f\psi_0)-\ppom v_0\in L^q(\Om)$  for all $q\in(1,n)$. This completes the proof.
\end{proof}

We conclude this section by recalling a stability result concerning  the large time behaviour of  solutions of the Navier-Stokes system in an exterior domain.

\begin{theorem}\label{stability}
Let $u_0, \widetilde u_0\in L^{n,\infty}_\sigma (\Om)$.
 There exists $\ep>0$ such that if $\|u_0\|_{L^{n,\infty}(\Om)}<\ep$ and   $\|\widetilde u_0\|_{L^{n,\infty}(\Om)}<\ep$, then the global small solutions $u$ and $\widetilde u$ of the exterior Navier-Stokes problem 
 \eqref{nsext} with  initial data $u_0$ and  $\widetilde u_0$ verify the following stability result. We have 
\begin{equation*}
t^{\frac12-\frac n{2p}}\nlp p{u(t)-\widetilde u(t)}\to 0  
\qquad\text{as}\quad t\to\infty \quad \forall p\in (n,\infty)
\end{equation*}
 if and only if 
\begin{equation*}
t^{\frac12-\frac n{2p}}\nlp p{e^{-tA}(u_0-\widetilde u_0)}\to 0  
\qquad\text{as}\quad t\to\infty\quad \forall p\in (n,\infty).
\end{equation*}
\end{theorem}

In this theorem, 
the global existence of the solutions was proved by  Kozono \& Yamazaki \cite{KY95} and the stability part is shown in \cite{MR2143356}.

Thus, in particular, 
 as long as the initial datum $v_0\in L^{n,\infty}(\R^n)$ is small,  Lemma \ref{restrictions} combined with Corollary \ref{cor:dec} and Theorem \ref{stability}
imply that the Navier-Stokes solutions in the exterior domain $\Omega$,
supplemented with the initial data  
$ \P (v_0\bigl|_{\Om})$ and $\dive(f\psi_0)$, have the same large time asymptotics.

An analogous result on the large time behaviour of solutions of semilinear parabolic equations with a scaling property was proved in \cite{MR1689406} in the case of the whole space $\R^n$.

\section{Asymptotics for small data in weak $L^n$-spaces}
\label{sect:withsmallu}

\subsection*{Statement of the  results}

Now we  show that, for  certain small initial conditions  in the whole space
   $\R^ n $ and in the exterior domain $\Omega$, the corresponding solutions of the Navier-Stokes system have the same large time asymptotics.

 Let $v_0\in L^{ n ,\infty}_\sigma (\R^ n )$ be  sufficiently small and
denote by  $v$  the unique global-in-time solution of the Navier-Stokes equations  in $\R^ n $ with  initial velocity $v_0$:
\begin{align}\label{veq}
\partial_t v-\Delta v+v\cdot\nabla v&=-\nabla p \qquad\text{for}\quad x\in\R^ n,\\
\dive v&=0,\\
v(0,\cdot)&=v_0.\label{veq-ini}   
\end{align}

We denote by $\psi_0$ 
the skew-symmetric matrix  
constructed as in Lemma \ref{defpsic}. In particular,
$v_0=\dive\psi_0$ and  $\int_{B_R}\psi_0(x)\,dx=0$.
Then, the vector field 
\begin{equation}\label{vb0}
  \vb_0=\dive(f\psi_0),
\end{equation}
where $f$ is our fixed cut-off function,
is divergence free and vanishes on $\Gamma=\partial\Omega$. 
Moreover, 
if $v_0$ is small in $L^{n,\infty}_\sigma(\R^n)$, then 
$u_0=\vb_0\bigl|_\Om	$ is also small in $L^{n,\infty}_\sigma (\Om),$
as an immediate consequence of Lemma \ref{lem:mapping}.

We are going to compare the large time behaviour of the solution $v=v(t,x)$ 
of the whole space problem \eqref{veq}-\eqref{veq-ini}
with the solution  $u=u(t,x)$  of the following exterior 
Navier-Stokes problem 
\begin{align}
\partial_t u-\Delta u+u\cdot\nabla u&=-\nabla q 
\qquad\text{for}\quad x\in\Omega,\label{ueq}\\
\dive u&=0,\\
u(0,\cdot)&=u_0\equiv  \vb_0\bigl|_\Om,\label{ueq-ini}    
\end{align}
where $\vb_0$ is defined in \eqref{vb0}.

If $\|v_0\|_{L^{n,\infty}(\R^n)}$ is sufficiently small, say
$\|v_0\|_{L^{n,\infty}(\R^n)}\leq\ep_1$ for suitable small $\ep_1>0$, then 
a solution $v=v(x,t)$ of the Cauchy problem \eqref{veq}-\eqref{veq-ini}  
exists globally-in-time 
and for every $n<p<\infty$ it satisfies  the following bounds
\begin{equation}\label{ineq:decay}
\begin{split}
\nlpr p{v(t)}\leq C\ep_1 t^{\frac n{2p}-\frac12}, &\qquad\nlpr p{\nabla v(t)}\leq C\ep_1 t^{\frac n{2p}-1}, \\
\nlpr p{\partial_tv(t)}\leq C\ep_1 t^{\frac n{2p}-\frac32},&\qquad
\nlpr p{\Delta v(t)}\leq C\ep_1 t^{\frac n{2p}-\frac32}, 
\end{split}
\end{equation}
for all $t>0$
(see {\it e.g.}~\cite{GS03}). The constant $C$ depends on $p$ but not on $\ep_1$.
Moreover, since by Lemma \ref{lem:mapping}, there exists a constant $C>0$ such that 
$\nlp p{u_0} =\nlp p{\vb_0}\leq C\varepsilon_1$, 
for sufficiently small $\ep_1>0$ there exists a global-in-time solution $u=u(x,t)$ of the exterior 
problem \eqref{ueq}-\eqref{ueq-ini} which satisfies the estimate
\begin{equation*}
 \nlp p{u(t)}\leq C\ep_1 t^{\frac n{2p}-\frac12}
\end{equation*}
for some constant $C=C(p)>0$ and for all $t>0$ (see \cite{KY95}).

In the main theorem of this section, we  show that both solutions 
$u$ and $v$ have the same large time behaviour.

\begin{theorem}\label{withsmallu}
There exists a constant $\ep_1>0$ with the following property. Let $v_0\in L^{n,\infty}_\sigma(\R^n)$ with $\|v_0\|_{L^{n,\infty}(\R^n)}\leq\ep_1$ and define $u_0=\vb_0\bigl|_\Om$ obtained from $v_0$ 
via  the cut-off procedure  \eqref{vb0}. 
Then, there exist  $v=v(t,x)$ and $u=u(t,x)$  unique global-in-time solutions of the Cauchy problem
\eqref{veq}-\eqref{veq-ini} and of
the exterior problem \eqref{ueq}-\eqref{ueq-ini}, respectively. Moreover, for every $ n <p<\infty$ we have 
\begin{equation*}
\lim_{t\to\infty}t^{\frac12-\frac{n}{2p}}\nlp p{u(t)-v(t)}=0  .
\end{equation*}
\end{theorem}

We might want to have a similar result for other initial conditions
$u_0$. This is easily obtained via the stability result stated in Theorem  \ref{stability}. For example, the same result holds true if we choose $u_0=\P(v_0\bigl|_\Om)$ instead of $u_0=\vb_0\bigl|_\Om$. Indeed, by Lemma \ref{restrictions} and the decay estimates for the Stokes operator given in 
\eqref{S1}, we have that 
\[
\lim_{t\to\infty}t^{\frac12-\frac n{2p}}\nlp p{e^{-tA}(\P(v_0\bigl|_\Om)-\dive(f\psi_0))}= 0 
\]
for every $p>n.$
Moreover, if  $\| v_0 \|_{L^{n,\infty}(\R^n)}\leq \ep$,
by the continuity of the Leray projector, we obtain that
 $\nlpw n{ \P(v_0\bigl|_\Om) }\leq C\ep$.
Hence, if $\varepsilon>0$ is sufficiently small, it suffices to apply
Theorem  \ref{stability}.

Conversely, we might want to fix $u_0$ and to construct $v_0$ instead of the other way round. This is also made possible by our results.   Given $u_0\in L^{n,\infty}_\sigma(\Omega)$, we may choose 
 $v_0$ to be the extension of $u_0$ to $\R^n$ with zero values outside $\Om$ (or any other extension that preserves the divergence free condition and the smallness of the $L^{n,\infty}$-norm). 
Indeed, for such an extension  we have  that  $u_0=\P(v_0\bigl|_\Om)$ and the above applies provided that $u_0$ is sufficiently small in $L^{n,\infty}$.

In fact, this property 
 can be applied to all initial data comparable at infinity. More precisely, we have the following result.

\begin{theorem}\label{corwithsmallu} Let $v_0\in L^{n,\infty}_{\sigma}(\R^n)$ and  $u_0\in L^{n,\infty}_{\sigma}(\Om)$ be comparable at infinity. 
There exists $\ep_1>0$ such that if $\|v_0\|_{L^{n,\infty}(\R^n)}\leq\ep_1$ and $\|u_0\|_{L^{n,\infty}(\Om)}\leq\ep_1$, then for every $p\in  (n, \infty)$, the corresponding solutions 
of the Cauchy problem
\eqref{veq}-\eqref{veq-ini} and 
the exterior problem \eqref{ueq}-\eqref{ueq-ini}, respectively,
satisfy
$
\lim\limits_{t\to\infty}t^{\frac12-\frac{n}{2p}}\nlp p{u(t)-v(t)}=0  .
$
\end{theorem}

\begin{proof} 

Let us denote by $\widetilde u$ the solution of the Navier-Stokes equations on $\Om$ with initial velocity $\widetilde u_{0}= \bar v_{0}\vert_{\Om}$, where $\bar v_{0}$ is constructed in \eqref{vb0}.  
Using Lemma \ref{lem:mapping} we have  $\|\widetilde u_0\|_{L^{n,\infty}(\Om)}<C\ep$. 
Hence, if $\ep_1>0$ is sufficiently small,
by Theorem \ref{withsmallu}, we obtain that
$ \lim\limits_{t\to\infty}t^{\frac12-\frac{n}{2p}}\nlp p{\widetilde u(t)-v(t)}=0$
for each $p\in (n,\infty)$.

Next, we use the decomposition
\[
u_{0}- \widetilde u_{0}= \P(u_{0}- \widetilde u_{0})=\P(u_{0}-v_{0}\bigl|_\Om)+\P(v_{0}\bigl|_\Om-\widetilde u_{0}).
\]
Since, by Definition \ref{def:eq}, 
we have that   $u_0-v_0 \in L^{q_{0}}(\Om)$ for some $q_{0}\in (1,n]$, we obtain
 $\P(u_{0}-v_{0}\bigl|_\Om)\in L^{q_{0}}(\Om)$. 
Moreover,  recalling that $\widetilde u_{0}=\vb_0\bigl|_\Om$  we have that 
$ \P(v_{0}\bigl|_\Om-\widetilde u_{0})\in L^q(\Om)$ for every $q\in (1,n)$. 
Hence, either by decay estimate \eqref{S1}
or by Corollary \ref{cor:dec} (if $q_{0}=n$), we get that
$
\lim\limits_{t\to\infty} t^{\frac12-\frac n{2p}}\nlp p{e^{-tA}(u_{0}- \widetilde u_{0})}= 0 .
$
Thus, if $\ep_1>0$  is sufficiently small, we may apply 
Theorem \ref{stability} to show $\lim\limits_{t\to\infty}t^{\frac12-\frac{n}{2p}}\nlp p{u(t)-\widetilde u(t)}=0 $.
The triangle inequality completes the proof of Theorem \ref{corwithsmallu}.
\end{proof}

\subsection*{Proof of Theorem \ref{withsmallu}}
The remainder  of this section is devoted to the proof of Theorem \ref{withsmallu}. 
Here, we use the auxiliary vector field $\vb=\dive (f\psi)$, where $\psi=\psi(x,t)$ is the skew-symmetric 
matrix  obtained from the solution $v=v(x,t)$ as in Lemma \ref{defpsic}. From \eqref{ineq:decay} we infer that
\begin{equation}\label{v:bar:v}
t^{\frac12-\frac{n}{2p}}\nlp p{\vb(t)}\leq Ct^{\frac12-\frac{n}{2p}}\nlpr p{v(t)}\leq C(p)\ep_1
\end{equation}
for all $t>0$, each $p>n$, and some constant $C(p)>0$.

To prove Theorem \ref{withsmallu},
it suffices to show that for each $p\in (n,\infty)$ we have that 
\begin{equation}\label{estw}
\lim_{t\to\infty}t^{\frac12-\frac{n}{2p}}\nlp p{w(t)}=0, \qquad\text{where} \quad w=u-\vb.
\end{equation}
Indeed, 
Theorem \ref{withsmallu} will be a direct consequence of the inequality
 $$\nlp p{u-v}\leq\nlp pw+\nlp p{v-\vb},$$ 
where the difference
 $v-\vb$ is compactly supported, so  $\nlp p{v-\vb}\leq C\nlp r{v-\vb}$ for all  $r>p$. 
Thus, using
 the $L^r$-decay estimate of $v$ given in \eqref{ineq:decay}
together with \eqref{v:bar:v}
we obtain for $r>p$
\begin{equation*}
t^{\frac12-\frac{n}{2p}}\nlp p{v(t)-\vb(t)}\leq   Ct^{\frac12-\frac{n}{2p}}\big(\nlp r{v(t)}+\nlp r{\vb(t)}\big)
\leq Ct^{\frac n{2r}-\frac{n}{2p}}
\to 0\quad \text{as}\quad t\to\infty.
\end{equation*}

To prove the limit in 
 \eqref{estw}, we note that    the vector field
$w=u-\vb$ verifies the following system
\begin{equation}\label{eqw}
\partial_t w-\Delta w+w\cdot\nabla u+\vb\cdot \nabla w= -\nabla(q-fp)-F  \quad\text{in}\quad \Omega,\qquad w\vert_{t=0}=0,
\end{equation}
where  the forcing term $F$ is given by the following expression
\begin{equation}\label{decF}
F=p\nabla f+\partial_t\psi\nabla f+(f\Delta v-\Delta \vb)+(\vb\cdot\nabla\vb-f v\cdot\nabla v)\equiv F_1+F_2+F_3+F_4.
\end{equation}
Moreover, $w$ is divergence free and vanishes on the boundary.

In the following proposition, we state that \eqref{estw} is a consequence of the decay properties of  $F$.

\begin{proposition}\label{smallF}
If $\ep_1>0$ is sufficiently small and if for each  $p\in(n,\infty)$
\begin{equation}\label{as:inf:F}
 \lim_{t\to\infty} t^{\frac12-\frac{n}{2p}}\nlp p{\int_0^t e^{-(t-s)A}\P F(s)\,ds}=0, 
\end{equation}
 then for each $p\in(n,\infty)$ we have that 
$
\lim\limits_{t\to\infty} t^{\frac12-\frac{n}{2p}}\nlp p{w(t)}=0. 
$
\end{proposition}

\begin{proof}
Our reasoning is inspired by 
 the proof of the stability Theorem \ref{stability} from \cite{MR2143356}.

Because of the interpolation inequality 
$$
\nlp {p_2}{\cdot}\leq\nlp{p_1}\cdot^\al\nlp{p_3}\cdot^{1-\al}, \qquad \text{where}\quad 
\al=\frac{p_1p_3-p_1p_2}{p_2p_3-p_1p_2}\quad \text{and}\quad  p_1<p_2<p_3,
$$
and the fact that $t^{\frac12-\frac{n}{2p}}\nlp p{w(t)}$ is bounded,  we remark that it suffices to get the limit of the $L^p$-norm of $w(t)$ for only one $p\in(n,\infty)$. 
Let us consider   $p>2n$.

Applying the Leray projector $\P$ to \eqref{eqw} and using the Duhamel formula we obtain that
\begin{equation}\label{duhamel}
w(t)=-\int_0^t e^{-(t-s)A}\P\dive(w\otimes u+\vb\otimes w)(s)\,ds-\int_0^t e^{-(t-s)A}\P F(s)\,ds.
\end{equation} 
The function
\begin{equation*}
h(t)\equiv  t^{\frac12-\frac{n}{2p}}\nlp p{w(t)}
\end{equation*}
is bounded on $(0,\infty)$ due to estimates \eqref{v:bar:v} and \eqref{ineq:decay}.
Hence, using the  decay estimate  \eqref{S3} with $q=p/2\geq n$, we may bound 
\begin{align*}
\nlp p{ \int_0^t e^{-(t-s)A}\P\dive(w\otimes u+\vb\otimes w)(s)\,ds} \hskip -4cm &\\
&\leq C\int_0^t (t-s)^{-\frac12-\frac n{2p}}\nlp{\frac p2}{ (w\otimes u+\vb\otimes w)(s)}\,ds\\
&\leq C\int_0^t (t-s)^{-\frac12-\frac n{2p}}\nlp p{w(s)}(\nlp p{u(s)}+\nlp p{\vb(s)})\,ds\\
&\leq C\ep_1\int_0^t (t-s)^{-\frac12-\frac n{2p}}s^{-1+\frac n{p}}h(s)\,ds\\
&= C\ep_1 t^{-\frac12+\frac n{2p}}\int_0^1 (1-\tau)^{-\frac12-\frac n{2p}}\tau^{-1+\frac n{p}}h(t\tau)\,d\tau.
\end{align*}
Thus, computing  the $L^p$-norm in \eqref{duhamel} we have 
\begin{equation*}
h(t)\leq   C\ep_1\int_0^1 (1-\tau)^{-\frac12-\frac n{2p}}\tau^{-1+\frac n{p}}h(t\tau)\,d\tau+t^{\frac12-\frac{n}{2p}}\nlp p{\int_0^t e^{-(t-s)A}\P F(s)\,ds}.
\end{equation*}
Now, we apply the $\limsup\limits_{t\to\infty}$ to both sides of the previous inequality. 
By the Lebesgue dominated convergence theorem  and \eqref{as:inf:F},
we infer that
\begin{align*}
\limsup  _{t\to\infty}h(t)
\leq &C\ep_1  \limsup  _{t\to\infty}h(t) \int_0^1 (1-\tau)^{-\frac12-\frac n{2p}}\tau^{-1+\frac n{p}}\,d\tau\\
&\qquad +\limsup  _{t\to\infty}t^{\frac12-\frac{n}{2p}}\nlp p{\int_0^t e^{-(t-s)A}\P F(s)\,ds}\\
=&C\ep_1 \limsup  _{t\to\infty}h(t).
\end{align*}
Choosing $\ep_1>0$ such that  $C\ep_1<1$, we conclude that 
$\limsup  \limits_{t\to\infty}h(t)=0$.
This completes the proof of Proposition \ref{smallF}.
\end{proof}

To establish \eqref{as:inf:F}, we use the decomposition 
of $F$  in \eqref{decF} to write
\begin{equation*}
 \int_0^t e^{-(t-s)A}\P F(s)\,ds
=\sum_{j=1}^4 \int_0^t e^{-(t-s)A}\P F_j(s)\,ds 
\equiv \sum_{j=1}^4I_j,
\end{equation*}
and we  show that 
\begin{equation*}
\lim_{t\to\infty}t^{\frac12-\frac n{2p}}\nlp p{I_j(t)  }=0\qquad\text{for each}\quad  j\in\{1,2,3,4\}
\quad\text{and for every}\quad p\in (n,\infty).    
\end{equation*}

As $\supp(\nabla f)\subset B_R$, we note that $\vb=v$ outside the ball
$B_R$. Using also the fact that $f=1$  outside $B_{2R/3}$, we obtain immediately ({\it cf.}~\eqref{decF})
 that the terms $F_1,\dots,F_4$ are compactly supported in the ball $B_R$. 
In order to estimate each term $I_j$, we prove  the following lemma:

\begin{lemma}\label{lemmadec}
 Let $q\in(1,p]$, $r\in[q,\infty)$, $0\leq a<b\leq t$ and consider a sufficiently
regular  function $g(t,x)$ supported in $\R_+\times B_R$. There exists a constant $C=C(R,p,q,r)>0$ such that 
\begin{equation*}
\nlp p{ \int_a^b e^{-(t-s)A}\P g(s)\,ds }\leq C\int_a^b  (t-s)^{\frac{n}{2p}-\frac{n}{2q}} \nlpc r{  g(s)}\,ds. 
\end{equation*}
\end{lemma}
\begin{proof}
 Using the decay estimate \eqref{S1}
 and recalling that the Leray projector $\P$ is bounded on $L^q(\Omega)$ for all $1<q<\infty$, we may estimate
\begin{align*}
\nlp p{ \int_a^b e^{-(t-s)A}\P g(s)\,ds }
&\leq C \int_a^b (t-s)^{\frac{n}{2p}-\frac{n}{2q}}\nlp q{  g(s)}\,ds\\
&\leq C \int_a^b (t-s)^{\frac{n}{2p}-\frac{n}{2q}}\nlpc r{  g(s)}\,ds.
\end{align*}
\end{proof}

In the remainder of this section, we use exponents $p,q,r$ satisfying
$$\frac n2<q\leq p\quad\text{and}\quad  q\leq r<\infty.$$ 
The values of $q$ and $r$ will be conveniently chosen later on
and can possibly change from a term to another.
Sometimes, we may use the notation $r_1$ or $r_2$ instead of $r$, when different  additional assumptions on $r$ are imposed on $r_1$ and $r_2$
and if there is a possibility of confusion.  We always assume that $q\leq r_1,r_2<\infty$. Observe also that, under these  assumptions,  we have that  $\frac n{2p}-\frac n{2q}>-1$.

\bigskip

{\it Estimate of $I_1$.} Using Lemma \ref{lemmadec} we may bound
\begin{align*}
t^{\frac12-\frac{n}{2p}}\nlp p{I_1(t)  }
&\leq C t^{\frac12-\frac{n}{2p}}\int_0^t (t-s)^{\frac{n}{2p}-\frac{n}{2q}}\nlpr r{  p(s)}\,ds.
\end{align*}
It is well-known that the pressure term in \eqref{veq} can be expressed as 
$
p=-\sum_{i,j}\partial_i\partial_j\Delta^{-1}(v_i v_j),
$
where  the operator $\partial_i\partial_j\Delta^{-1}$ is bounded on $L^\al(\R^n)$ for every $1<\al<\infty$.
Thus, using the H\"older inequality we obtain the estimate 
\begin{equation*}
\nlpr r{  p(s)}\leq C\nlpr{2r}{v(s)}^2\leq Cs^{\frac n{2r}-1}
\end{equation*}
which leads to the inequality
\begin{equation*}
 t^{\frac12-\frac{n}{2p}}\nlp p{I_1(t)  }
\leq Ct^{\frac12-\frac n{2q}+\frac n{2r}}. 
\end{equation*}
Clearly, the right-hand side goes to 0 as $t\to\infty$ if $q<n$ and $r$ is sufficiently large.

\bigskip

{\it Estimate of $I_2$.} We decompose
\begin{align*}
I_2&=   \int_0^t e^{-(t-s)A}\P [\partial_s\psi(s)\nabla f]\,ds \\
&=\int_0^{\frac t2} e^{-(t-s)A}\P [\partial_s\psi(s)\nabla f]\,ds +\int_{\frac t2}^t e^{-(t-s)A}\P [\partial_s\psi(s)\nabla f]\,ds \\
&\equiv I_{21}+I_{22}.
\end{align*}
Using Lemma \ref{lemmadec} we have that
\begin{equation*}
t^{\frac12-\frac{n}{2p}}\nlp p{I_{22}(t)  }
\leq C t^{\frac12-\frac{n}{2p}}\int_{\frac t2}^t (t-s)^{\frac{n}{2p}-\frac{n}{2q}}\nlpc r{ \partial_s\psi(s) }\,ds.
\end{equation*}
Since $\psi$ has zero average on the ball $B_R$, so does $\partial_t \psi$. By the Poincaré inequality applied on $B_R$
and recalling \eqref{ineq:decay},  we have that
\begin{equation*}
\nlpc r{ \partial_s\psi(s) }\leq C\nlpc r{ \partial_s\nabla\psi(s) }\leq C  \nlpr r{ \partial_sv(s) }\leq Cs^{\frac n{2r}-\frac32},
\end{equation*}
where we have used \eqref{boundS} applied to $\partial_t\psi$. 
Thus
\begin{equation*}
t^{\frac12-\frac{n}{2p}}\nlp p{I_{22}(t)  }\leq C   t^{\frac12-\frac{n}{2p}}\int_{\frac t2}^t (t-s)^{\frac{n}{2p}-\frac{n}{2q}} s^{\frac n{2r}-\frac32}\,ds\leq C t^{\frac n{2r}-\frac n{2q}}\to\infty \quad \text{as }t\to\infty
\end{equation*}
provided that $q<r$.

To bound $I_{21}$, we integrate  by parts
\begin{equation}\label{i21}
\begin{split}
I_{21}&=  \int_0^{\frac t2} e^{-(t-s)A}\partial_s\P [\psi(s)\nabla f]\,ds \\
&=e^{-\frac t2A}\P [\psi(t/2)\nabla f]-e^{-tA}\P [\psi_0\nabla f] 
+\int_0^{\frac t2} Ae^{-(t-s)A}\P [\psi(s)\nabla f]\,ds .
\end{split}
\end{equation}
The $L^p$-norm of the first term on the right-hand side is easily bounded by 
\begin{equation*}
\begin{split}
\nlp p{ e^{-\frac t2A}\P [\psi(t/2)\nabla f]}&\leq C \nlpc p{\psi(t/2)}\leq  C \nlpc r{\psi(t/2)}\\
&\leq  C \nlpc r{\nabla\psi(t/2)}\leq C \nlpr r{v(t/2)}\leq Ct^{\frac n{2r}-\frac12}.
\end{split}
\end{equation*}
The second right-hand side term  of \eqref{i21} can be estimated as follows
\begin{equation*}
\nlp p{ e^{-tA}\P [\psi_0\nabla f] }\leq Ct^{\frac{n}{2p}-\frac{n}{2q}}\nlpc q{\psi_0}.
\end{equation*}

Concerning the last term in \eqref{i21}, we first note that
\begin{equation*}\begin{split}
\nlp p{\P [\psi(s)\nabla f]}&\leq 
C\nlpc p{\psi(s)}\leq 
C\nlpc r{\psi(s)}\\
&\leq 
C\nlpc r{\nabla\psi(s)}\leq 
C\nlpr r{v(s)}\leq 
C s^{\frac n{2r}-\frac12}.
\end{split}
\end{equation*}
The Stokes semigroup estimates from  Proposition \ref{ellpest} give
\begin{align*}
\nlp p{ \int_0^{\frac t2} Ae^{-(t-s)A}\P [\psi(s)\nabla f]\,ds }
&\leq C  \int_0^{\frac t2} (t-s)^{-1}\nlp p{\P [\psi(s)\nabla f]}\,ds \\
&\leq C  \int_0^{\frac t2} (t-s)^{-1}s^{\frac n{2r}-\frac12}\,ds= C  t^{\frac n{2r}-\frac12}.
\end{align*}

We conclude from the previous estimates  that
\begin{equation*}
t^{\frac12-\frac{n}{2p}}\nlp p{I_2(t)}  
\leq t^{\frac12-\frac{n}{2p}}\nlp p{I_{21}(t)} + t^{\frac12-\frac{n}{2p}}\nlp p{I_{22}(t)}   
\to 0
\quad \text{as}\quad t\to\infty
\end{equation*}
provided  that $q<n$ and $r>p$.

\bigskip

{\it Estimate of $I_4$.} We have that $F_4=\vb\cdot\nabla\vb-f v\cdot\nabla v=\vb\cdot(\nabla\vb-f\nabla v)+f (\vb-v)\cdot\nabla v$, where  both terms $\vb\cdot(\nabla\vb-f\nabla v)$ and $f (\vb-v)\cdot\nabla v$ are supported in $B_R$. Now, we decompose 
\begin{equation*}
I_4= \int_0^t e^{-(t-s)A}\P[ f (\vb-v)\cdot\nabla v](s)\,ds+\int_0^t e^{-(t-s)A}\P[ \vb\cdot(\nabla\vb-f\nabla v)](s)\,ds\equiv I_{41}+I_{42}.
\end{equation*}
 First, we  use Lemma \ref{lemmadec} to bound
\begin{equation*}
 \nlp p{I_{41}}\leq C \int_0^t (t-s)^{\frac{n}{2p}-\frac{n}{2q}}\nlpc {r_1} {(\vb-v)\cdot\nabla v(s)}\,ds.
\end{equation*}
Lemma \ref{lem:mapping} and relation \eqref{ineq:decay} imply that
\begin{equation*}
 \nlpr {r_1} {(\vb-v)\cdot\nabla v}\leq C\nlpr {2{r_1}}v \nlpr {2{r_1}}{\nabla v}\leq Ct^{\frac n{2{r_1}}-\frac32} 
\end{equation*}
so
\begin{equation*}
 \nlp p{I_{41}}\leq C \int_0^t (t-s)^{\frac{n}{2p}-\frac{n}{2q}}s^{\frac n{2{r_1}}-\frac32} \,ds\leq Ct^{\frac n{2p}-\frac n{2q}+\frac n{2{r_1}}-\frac12}.
\end{equation*}
Assuming that $\frac n{2{r_1}}-\frac32>-1$ (\textit{i.e.} ${r_1}<n$) and   $q<{r_1}$, it follows that
\begin{equation*}
t^{\frac12-\frac{n}{2p}}\nlp p{I_{41}}\leq Ct^{-\frac n{2q}+\frac n{2{r_1}}}\to0  
\qquad \text{as}\quad t\to\infty.
\end{equation*}

Using  again Lemma \ref{lemmadec}  we  bound
\begin{equation*}
 \nlp p{I_{42}}\leq C \int_0^t (t-s)^{\frac{n}{2p}-\frac{n}{2q}}\nlpc {r_2} {\vb\cdot(\nabla\vb-f\nabla v)(s)}\,ds.
\end{equation*}
Since $\nabla\vb=f\nabla v+v\otimes\nabla f+\nabla\psi\nabla f+\psi\nabla^2 f$, by relation  \eqref{boundS},
 we can further estimate
\begin{align*}
\nlpc {r_2} {\vb\cdot(\nabla\vb-f\nabla v))}  
&\leq \nlpr {2{r_2}}{\vb}\nlpr{2{r_2}}{v\otimes\nabla f+\nabla\psi\nabla f+\psi\nabla^2 f }\\
&\leq  C\nlpr {2{r_2}}{\vb}(\nlpr{2{r_2}}{v} +\nlpc{2{r_2}}{\psi})\\
&\leq  C\nlpr {2{r_2}}v^2\leq Ct^{\frac n{2{r_2}}-1}.
\end{align*}
Thus, we infer that
\begin{equation*}
\nlp p{I_{42}}
\leq C  \int_0^t (t-s)^{\frac{n}{2p}-\frac{n}{2q}}s^{\frac n{2{r_2}}-1}\,ds
\leq Ct^{\frac{n}{2p}-\frac{n}{2q}+\frac n{2{r_2}}},
\end{equation*}
and therefore
$
t^{\frac12-\frac{n}{2p}}\nlp p{I_{42}}\leq Ct^{\frac12-\frac n{2q}+\frac n{2{r_2}}}\to0  
$
 as $t\to\infty$ provided that   $q<n$ and ${r_2}$ is sufficiently large.

\bigskip

{\it Estimate of $I_3$.} Finally, we deal with the remaining term $I_3$. Since
\begin{equation*}
F_3=f\Delta v-\Delta\dive(f\psi)=  f\Delta v-\Delta(fv)-\Delta(\psi\nabla f)
=-v\Delta f-2\nabla f \nabla v  -\Delta(\psi\nabla f),
\end{equation*}
we obtain
\begin{align*}
I_3&= \int_0^t e^{-(t-s)A}\P F_3(s)\,ds \\
&=-\int_0^t e^{-(t-s)A}\P (v\Delta f)(s)\,ds -2\int_0^t e^{-(t-s)A}\P (\nabla f \nabla v)(s)\,ds 
-\int_0^t e^{-(t-s)A}\P\Delta(\psi\nabla f)(s) \,ds \\
&\equiv I_{31}+I_{32}+I_{33}.
\end{align*}

To bound  the middle term, it suffices to  use Lemma \ref{lemmadec} in the following way
\begin{align*}
t^{\frac12-\frac{n}{2p}}\nlp p{I_{32}(t)}
&\leq Ct^{\frac12-\frac{n}{2p}}\int_0^t (t-s)^{\frac{n}{2p}-\frac{n}{2q}}\nlpc r{ \nabla v(s)}\,ds\\
&\leq Ct^{\frac12-\frac{n}{2p}}\int_0^t (t-s)^{\frac{n}{2p}-\frac{n}{2q}}s^{\frac n{2r}-1}\,ds\\
&\leq Ct^{\frac12-\frac n{2q}+\frac n{2r}}\to0 \quad \text{as}\quad t\to\infty,
\end{align*}
 provided that   $q<n$ and $r$ is sufficiently large.

Now, we are going to  estimate  $I_{33}$. 
 It is well-known that for every vector field $w$ 
(not necessarily neither divergence free nor tangent to the boundary), 
 the quantity 
$w-\P w$ is a gradient: $w-\P w=\nabla g$ for some scalar function $g$. 
It is also known that the Leray projection $\P$ of such a vector field 
 vanishes. Therefore,
 $\P\Delta w-\P\Delta\P w=\P\Delta(w-\P w)=\P\Delta\nabla g=\P\nabla\Delta g=0$ so that $\P\Delta w=\P\Delta\P w=A\P w$. Using this observation we may write
\begin{align*}
I_{33}&=-\int_0^t e^{-(t-s)A}\P\Delta(\psi\nabla f)(s) \,ds \\
&=\int_0^{t-\frac1t}Ae^{-(t-s)A}\P(\psi\nabla f)(s) \,ds
-\int_{t-\frac1t}^t e^{-(t-s)A}\P\Delta(\psi\nabla f)(s) \,ds\\
&\equiv I_{331}+I_{332}.
\end{align*}
We  apply  Lemma \ref{lemmadec} with $p=q$ and recall relation \eqref{boundS} to bound
\begin{align*}
\nlp p{I_{331}}
&\leq \int_0^{t-\frac1t} (t-s)^{-1}\nlpc r{\psi(s)}\,ds\\
&= \int_0^{t-\frac1t} (t-s)^{-1}\nlpc r{v(s)}\,ds\\
&\leq \int_0^{t-\frac1t} (t-s)^{-1}s^{\frac n{2r}-\frac12}\,ds
\leq Ct^{\frac n{2r}-\frac12}\ln(2+t).
\end{align*}
 We conclude that
$
t^{\frac12-\frac{n}{2p}}  \nlp p{I_{331}}\leq Ct^{\frac n{2r}-\frac n{2p}}\ln(2+t)\to0
$
as $t\to\infty$ provided that   $r>p$.

To estimate  $I_{332}$, we use again Lemma \ref{lemmadec} with $p=q$:
\begin{equation*}
\nlp p{I_{332}}\leq  \int_{t-\frac1t}^t \nlpc r{\Delta(\psi\nabla f)(s) }\,ds. 
\end{equation*}
Clearly, we have 
\begin{equation*}
\begin{split}
\nlpc r{\Delta(\psi\nabla f) (s)}&\leq C(\nlpc r{\psi}+\nlpc r{\nabla \psi}+\nlpc r{\Delta\psi})\\  
&\leq C(\nlp r{v}+\nlp r{\nabla v})  
\leq C(s^{\frac n{2r}-\frac12}+s^{\frac n{2r}-1})\leq Cs^{\frac n{2r}-\frac12}
\end{split}
\end{equation*}
for all $s\in(t-1/t,t)$ (notice that we may assume $t\geq2$). Therefore,
\begin{equation*}
t^{\frac12-\frac{n}{2p}} \nlp p{I_{332}}\leq  
t^{\frac12-\frac{n}{2p}} \int_{t-\frac1t}^t s^{\frac n{2r}-\frac12}\,ds \leq C
t^{\frac12-\frac{n}{2p}}t^{\frac n{2r}-\frac32}
=Ct^{\frac n{2r}-\frac n{2p}-1}\to0 \quad \text{as}\quad t\to\infty.
\end{equation*}

It remains to estimate the term $I_{31}$ which we decompose in the following way
{\allowdisplaybreaks
\begin{align*}
 I_{31}
=&-\int_0^t e^{-(t-s)A}\P (v\Delta f)(s)\,ds \\
=&\int_0^t e^{-(t-s)A}\P (v\Delta (1-f))(s)\,ds \\
=&\int_0^{\frac1t} e^{-(t-s)A}\P (v\Delta (1-f))(s)\,ds + \int_{\frac1t}^t e^{-(t-s)A}\P (v\Delta (1-f))(s)\,ds \\
=&\int_0^{\frac1t} e^{-(t-s)A}\P (v\Delta (1-f))(s)\,ds + \int_{\frac1t}^t e^{-(t-s)A}\P \Delta (v(1-f))(s)\,ds \\
& +2 \int_{\frac1t}^t e^{-(t-s)A}\P (\nabla v \cdot\nabla f)(s)\,ds - \int_{\frac1t}^t e^{-(t-s)A}\P (\Delta v (1-f))(s)\,ds \\
\equiv& I_{311}+I_{312}+I_{313}+I_{314}.
\end{align*}
}

As the function $1-f$ is supported on the ball $B_R$, the terms 
$I_{312}$ and $I_{313}$ can be treated in the same way as the integrals $I_{33}$  and $I_{32}$.
Next, we use Lemma \ref{lemmadec} to deal with $I_{311}$:
\begin{align*}
t^{\frac12-\frac{n}{2p}}\nlp p {I_{311}}&\leq Ct^{\frac12-\frac{n}{2p}}\int_0^{\frac1t} (t-s)^{\frac{n}{2p}-\frac{n}{2q}}\nlpr r{v(s)}\,ds  \\
&\leq Ct^{\frac12-\frac{n}{2p}}\int_0^{\frac1t} (t-s)^{\frac{n}{2p}-\frac{n}{2q}}s^{\frac n{2r}-\frac12}\,ds  
\leq Ct^{-\frac n{2q}-\frac n{2r}}\to0
\end{align*}
as $t\to\infty$. Moreover, we study $I_{314}$ 
 in a similar manner 
\begin{align*}
t^{\frac12-\frac{n}{2p}}\nlp p {I_{314}}&\leq Ct^{\frac12-\frac{n}{2p}}\int_{\frac1t}^t (t-s)^{\frac{n}{2p}-\frac{n}{2q}}\nlpr r{\Delta v(s)}\,ds  \\
&\leq Ct^{\frac12-\frac{n}{2p}}\int_{\frac1t}^t (t-s)^{\frac{n}{2p}-\frac{n}{2q}}s^{\frac n{2r}-\frac32}\,ds 
\leq Ct^{1-\frac n{2q}-\frac n{2r}} \to0
\end{align*}
as $t\to\infty$ provided that   $r>n$ and $q$ is sufficiently close to $\frac n2$. 

This completes the proof of Theorem \ref{withsmallu}.

\section{Asymptotics of solutions for large data and $n\geq3$}
\label{sectdim3}

In the case $n\geq3$, Theorem \ref{corwithsmallu}
can be extended to a certain class of large initial data as follows.

\begin{theorem}\label{largedata}
Suppose that $n\geq3$. There exists $\ep>0$ such that if $u_0\in L^{n,\infty}(\Om)$ is divergence free, tangent to the boundary, and such that
\begin{equation*}
\limsup_{\la\to0}\la\mes\{x\in\Omega \,:\, |u_0(x)|>\la\}^{\frac1n}<\ep,
\end{equation*}
then there exists a global-in-time  solution $u$ of the Navier-Stokes equations on the exterior domain $\Om$ with initial velocity $u_0$ and such that
\begin{equation}\label{limsupex}
\limsup _{t\to\infty}t^{\frac12-\frac n{2p}}\nlp p{u(t)}<\infty.  
\end{equation}
Moreover, let $v_0\in L^{n,\infty}(\R^n)$ be divergence free and such that
\begin{equation*}
\limsup_{\la\to0}\la\mes\{x\in \R^n\,:\, |v_0(x)|>\la\}^{\frac1n}<\ep.
\end{equation*}
 Then the Navier-Stokes equations on $\R^n$ with initial velocity $v_0$ admit a  global-in-time 
 solution $v$ which satisfies a similar bound to the one in \eqref{limsupex}. 
Finally, if $u_{0}$ and $v_{0}$ are comparable at infinity {\rm (}{\it cf.}~Definition \ref{def:eq}{\rm )}, then 
the corresponding solutions $u$ and $v$ have the same asymptotic behaviour as $t\to\infty$ in our usual  sense:
$
\lim\limits\limits_{t\to\infty}t^{\frac12-\frac{n}{2p}}\nlp p{u(t)-v(t)}=0
$
for each $p\in (n,\infty)$.
\end{theorem}
% \begin{remark}
% It may be that the solution $u(t)$ does not belong to $ L^p(\Om)$ for all times $t$. However, if $t$ is sufficiently large, then necessarily $u(t)\in L^p(\Om)$. This is why in relation \eqref{limsupex} we have a $\limsup$ instead of a $\sup$. A similar remark holds true for $v$.
% \end{remark}

\begin{remark}
Even though the initial velocity $u_0$ is not small in $L^{n,\infty}(\Om)$ and is not square-integrable, it is not surprising that a weak solution exists globally. 
Our result is reminiscent of the result in \cite{MR968416} on
the global-in-time existence of weak solutions to the Navier-Stokes system 
that holds true for initial data in $L^p_\sigma(\R^n)$ with $2<p<n$.  
\end{remark}

\begin{proof}[Proof of Theorem \ref{largedata}.]
We borrow a method from \cite{BBIS11}. Using   \cite[Lemma 4.2]{BBIS11} we decompose $u_0=z_0+w_0$ where $z_0\in L^2(\Om)$ and $\|w_0\|_{L^{n,\infty}(\Om)}\leq\ep$. Note that even though the statement of that lemma is given for $n=3$, the proof goes through for $n\geq3$. Let $w$ solve the Navier-Stokes equations on 
the exterior domain $\Om$ with initial velocity $w_0$ and set $z=u-w$. Then $z$ solves
\begin{equation*}
\partial_t z-\Delta z+u\cdot\nabla z+z\cdot\nabla w=-\nabla p'.  
\end{equation*}
We multiply this equation by $z$ and integrate in space to obtain after an integration by parts:
\begin{equation*}
\partial_t\nlp2z^2+2\nlp2{\nabla z}^2=2\int_\Om z\cdot\nabla z\cdot w
\leq C\|z\|_{L^{\frac{2n}{n-2},2}(\Om)}\nlp2{\nabla z}\|w\|_{L^{n,\infty}(\Om)}  
\end{equation*}
where we used the H\"older inequality for Lorentz spaces.
Next, we recall the following Sobolev inequality in the Lorentz spaces
\begin{equation}\label{L:Sob}
  \|z\|_{L^{\frac{2n}{n-2},2}(\Om)}\leq C\nlp2{\nabla z}.
\end{equation}
Indeed, in the case of functions defined on $\R^n$, this relation follows from the Young inequality for convolution in Lorentz spaces and from the observation that $(-\Delta)^{-\frac12}$ is a convolution operator with a function homogeneous of degree $1-n$ which therefore belongs to $L^{\frac n{n-1},\infty}(\R^n)$. 
For $z$ vanishing on $\partial\Om$, we can extend it with zero values outside $\Om$. Denoting by $Ez$ this extension
and   applying the inequality known in $\R^n$ for $Ez$, we get
 \eqref{L:Sob} for  $z$ defined in the exterior domain.
Consequently, we infer that
\begin{equation*}
\partial_t\nlp2z^2+2\nlp2{\nabla z}^2
\leq C\nlp2{\nabla z}^2\|w\|_{L^{n,\infty}(\Om)}.   
\end{equation*}

Recall now that $w$ is a small solution. Therefore, if $\|w_0\|_{L^{n,\infty}(\Om)}$ is sufficiently small, 
\textit{i.e.} if $\ep>0$ is sufficiently small, then we  have that 
 $C\|w(t)\|_{L^{n,\infty}(\Om)}\leq 1$ for all $t$. Thus,
\begin{equation*}
\partial_t\nlp2z^2+\nlp2{\nabla z}^2
\leq 0,
\end{equation*}
so that $Ez\in L^\infty(\R_+;L^2(\R^n))\cap L^2(\R_+;\dot{H}^1(\R^n))$. 
These are of course just \textit{a priori} estimates, but a standard approximation procedure gives us the existence of a solution verifying these estimates. Since $Ez\in L^\infty(\R_+;L^2(\R^n))\cap L^2(\R_+;\dot{H}^1(\R^n))$, we infer that $Ez\in L^4(\R_+;\dot{H}^{\frac12}(\R^n))$. Therefore, for every $\gamma>0$, the set of times $t$ where $\|Ez(t)\|_{\dot{H}^{\frac12}(\R^n)}\geq\gamma$ has finite measure. Since $\dot{H}^{\frac12}(\R^n)\subset L^{n,\infty}(\R^n)$, we infer that the set of times where $\|z(t)\|_{L^{n,\infty}(\Om)}\geq\gamma$ has finite measure. 

A similar argument can be performed 
in the case of the solution 
 $v$ to the Navier-Stokes problem in the whole space $\R^n$. 
Indeed, we may decompose
 $v_0=Z_0+W_0$ where both $Z_{0}$ and $W_0$ are divergence free, such that $Z_0\in L^2(\R^n)$ and $\|W_0\|_{L^{n,\infty}(\R^n)}\leq\ep$ (\cite[Lemma 4.2]{BBIS11}). We denote by $W$ the solution of the Navier-Stokes equations on $\R^n$ with initial data $W_0$ and set $Z=v-W$. Similar estimates as for $z$ show that, for every  $\gamma>0$, the set of times where $\|Z(t)\|_{L^{n,\infty}(\R^n)}\geq\gamma$ has finite measure. 

We conclude  that there is a time $T>0$ where we have both 
inequalities $\|Z(T)\|_{L^{n,\infty}(\R^n)}\leq\ep$ and  $\|z(T)\|_{L^{n,\infty}(\Om)}\leq\ep$. Since by the theory of small solutions we also have that $\|W(t)\|_{L^{n,\infty}(\R^n)}\leq C\ep$ and  $\|w(t)\|_{L^{n,\infty}(\Om)}\leq C\ep$ for all $t\geq0$, we infer that $\|v(T)\|_{L^{n,\infty}(\R^n)}\leq C'\ep$ and  $\|u(T)\|_{L^{n,\infty}(\Om)}\leq C'\ep$ with $C'=C+1$. We choose $\ep$ sufficiently small such that the small data solutions, the stability result from Proposition \ref{stability},
 and the asymptotic result Theorem \ref{corwithsmallu} can be applied starting from both times $t=0$ and $t=T$. In particular, we can change if necessary the values of $u$ and $v$ starting from time $T$ with the values of the small solutions that can be constructed starting from $u(T)$ and $v(T)$. 

Since $u(T)-w(T)\in L^2(\Om)$, the stability result stated in Theorem \ref{stability} applied from time $t=T$ and the decay estimates of the Stokes operator given in Proposition \ref{ellpest} imply that for all $p\in(n,\infty)$
\begin{equation*}
\lim_{t\to\infty} t^{\frac12-\frac n{2p}}\nlp p{u(t)-w(t)}=0. 
\end{equation*}
 Similarly, for all $p\in(n,\infty)$ we have that
\begin{equation*}
\lim_{t\to\infty} t^{\frac12-\frac n{2p}}\nlpr p{v(t)-W(t)}= 0.
\end{equation*}
 Moreover, $w_{0}-W_{0}=(u_{0}-v_{0})- (z_{0}-Z_{0})$ with $u_{0}-v_{0}\in L^{q_{0}}(\Om)$ for some $q_{0}\in (1,n]$ and 
$z_{0}-Z_{0}\in L^2(\Om)$ so $w_{0}-W_{0}\in L^{q_{0}}(\Om)+ L^2(\Om)$. We cannot apply directly Theorem \ref{corwithsmallu} to $w$ and $W$ because $w_0$ and $W_0$ are not comparable at infinity in the sense of Definition \ref{def:eq}: we have $w_{0}-W_{0}\in L^{q_{0}}+ L^2$ instead of $w_{0}-W_{0}\in L^q$ for some $q\in (1,n]$. Nevertheless, the proof of  Theorem \ref{corwithsmallu} goes through in this case and we obtain that  for all $p\in(n,\infty)$ we have 
\begin{equation*}
\lim_{t\to\infty} t^{\frac12-\frac n{2p}}\nlp p{w(t)-W(t)}=0.  
\end{equation*}
Putting together the three previous relations completes the proof.
\end{proof}

\section{Asymptotics for large data and $n=2$} 
\label{sectdim2}

We assume throughout this section that $n=2$. In this case we are not able to show a result as general as in dimension $n\geq3$, because  $L^2(\R^2)$ is a critical space (\textit{i.e.} invariant with respect to the scaling of the Navier-Stokes equations). Here, we decompose the initial velocity as follows
\begin{equation*}
u_0=\tu_0+w_0, 
\end{equation*}
where  $\tu_0\in \L2s$ is arbitrarily large and $w_0$ is small in $L^{2,\infty}(\Omega)$. Let us recall now the following classical result on the $L^2$-decay of weak solutions of problem \eqref{nsext}.
\begin{theorem}[{Borchers \& Miyakawa  \cite[Thm.~1.2]{MR1158939}}]\label{thm:L2decay}
 For every $\wu_0\in \L2s$ there is a unique weak solution 
$\tu\in L^\infty((0,\infty);L^2(\Om))\cap L^2_{loc}([0,\infty);H^1(\Om))$
of the  initial-boundary value problem
 \eqref{nsext} with $\tu_0$ as initial velocity.
This solution satisfies  
$
  \lim\limits\limits_{t\to\infty}\|\tu(t)\|_{L^2}=0.
$
\end{theorem}

The aim of this section is to show the following theorem.
\begin{theorem}\label{dim2}
Let $\tu_0\in \L2s$  be arbitrary and let $T_\ep$ be a time such that $\nlpw2{\tu(T_\ep)}<\ep/3$, where $\tu$ is the unique global-in-time solution of the Navier-Stokes equations on $\Om$ with initial velocity $\tu_0$ and $\ep>0$ is the small constant from Theorem \ref{corwithsmallu} 
{\rm (}such a time $T_\ep$ exists thanks to Theorem \ref{thm:L2decay} and to the imbedding $L^2(\Omega)\subset L^{2,\infty}(\Omega)${\rm )}. Let $w_0\in L^{2,\infty}(\Om)$ be divergence free, tangent to the boundary, and such that
\begin{equation*} 
 \|w_0\|_{L^{2,\infty}(\Om)}<\frac\ep3 .
\end{equation*}
There exists a constant $K=K(\ep)$ such that if the following additional smallness condition holds true:
  \begin{equation}\label{wsmallbis}
    \|e^{-tA}w_0\|_{L^4((0,T_\ep)\times\Om)}\leq\frac1{K\nlp2{\tu_0}e^{K\nlp2{\tu_0}^4}}
  \end{equation}
then the Navier-Stokes equations on the exterior domain $\Om$ with 
the initial data $u_0=\tu_0+w_0$ has a unique global solution $u$. This solution verifies  that 
$
\sup_{t>0}t^{\frac12-\frac 1{p}}\nlp p{u(t)}<\infty.  
$

Moreover, let $v_0\in L^{2,\infty}(\R^2)$ be such that $\|v_0\|_{L^{2,\infty}(\R^2)}<\ep$ and such that $w_0$ and $v_0$ are comparable at infinity. Let $v$ be the unique global solution of the Navier-Stokes equations on $\R^2$ with initial data $v_0$. Then
$u$ and $v$ have the same asymptotic behaviour as $t\to\infty$ in the  sense of relation \eqref{intro:asymp}.
\end{theorem}

\begin{remark}
If $ \nlp2{\tu_0}$ is sufficiently small, then the initial data $u_0$ is small in $L^{2,\infty}(\Om)$ and we can directly apply  Theorem  \ref{corwithsmallu} to reach the desired conclusion without the need to assume that $e^{-tA}w_0\in L\loc^4(\R_+;L^4(\Om))$ and  condition \eqref{wsmallbis}.
\end{remark}

\begin{remark}\label{rem:d2:old}
The dependence on the parameter $T_\ep$ in the smallness condition \eqref{wsmallbis} is not explicit. 
Nevertheless, it can  be made explicit if $w_0\in L^a(\Om)$ for some $a>2$. Indeed, for $w_0\in L^a(\Om)$, the decay estimates from Proposition \ref{ellpest} imply that $\nlp4{e^{-tA}w_0}\leq C\nlp a{w_0}t^{\frac14-\frac1a}$ (note that we can assume that $a\leq4$ because we also have that that $w_0\in L^{2,\infty}(\Om)$ so, by interpolation, $w_0$ belongs to all intermediate spaces $L^b(\Omega)$ for all  $2<b\leq a$). Therefore  $\|e^{-tA}w_0\|_{L^4((0,T_\ep)\times\Om)}\leq C\nlp a{w_0} T_\ep^{\frac12-\frac1a}$. We conclude that the smallness condition \eqref{wsmallbis} is implied by the following smallness condition
\begin{equation*}
 \nlp a{w_0}\leq\frac1{CT_\ep^{\frac12-\frac1a}\nlp2{\tu_0}e^{C\nlp2{\tu_0}^4}}
\end{equation*}
for some large constant $C$. In particular, the result in \cite{ILK11} is a special case of Theorem \ref{dim2}.
\end{remark}

\begin{remark}
The condition \eqref{wsmallbis} requires in particular that $e^{-tA}w_0\in L^4_{loc}([0,\infty);L^4(\Om))$.
The optimal assumption required on $w_0$ in order to have this property is that $w_0$ belongs to the inhomogeneous Besov space  $B^{-\frac12}_{4,4}(\Om)$ (see \cite{amann_nonhomogeneous_2002,depauw_solutions_2001,MR1938147}). Therefore, Theorem \ref{dim2} can be reformulated in the following way: 
for every $\tu_0\in L^2(\Om)$ which is divergence free and tangent to the boundary, there exists $\ep_3=\ep_3(\Om)>0$ and $\ep_4=\ep_4(\tu_0)>0$ such that if $w_0\in L^{2,\infty}(\Om)\cap B^{-\frac12}_{4,4}(\Om) $, $\nlpw2{w_0}<\ep_3$, and $\|w_0\|_{  B^{-\frac12}_{4,4}(\Om)}<\ep_4$ then the conclusion of Theorem~\ref{dim2} holds true. 
\end{remark}
\begin{proof}[Proof of Theorem \ref{dim2}.]
The global  existence and the uniqueness of $u$ is proved in \cite[Theorem 4]{KY95}. However, the bound for $t^{\frac12-\frac 1{p}}\nlp p{u(t)}$ proved in \cite{KY95} is only local in time, \textit{i.e.} it is only shown that  $\sup\limits_{0<t<T}t^{\frac12-\frac 1{p}}\nlp p{u(t)}\leq C(T)$ for all $T<\infty$. Here we will show that this bound is also global, \textit{i.e.} we can have $T=\infty$.

Let $\wb=e^{-tA}w_0$ and $z=u-\tu-\wb $. Then, the vector field  $z$ is divergence free, vanishes on the boundary, and verifies the following system 
\begin{equation*}
\partial_t z-\Delta z+(z+\wb +\tu)\cdot\nabla  (z+\wb +\tu)-\tu\cdot\nabla\tu=- \nabla \overline p,
\end{equation*}
where $\overline p=p-p_{\tu}-p_{\wb} $, the scalar function  $p_{\tu}$ is the pressure corresponding to $\tu$, and $p_{\wb} $ is the pressure from the Stokes problem satisfied by $\wb $. We multiply the above relation by $z$ and integrate over $\Omega$ to obtain
\begin{align*}
\frac12\partial_t\nlp2z^2+\nlp2{\nabla z}^2
&=\int_\Om  \tu\cdot\nabla\tu\cdot z-\int_\Om (z+\wb +\tu)\cdot\nabla  (z+\wb +\tu)\cdot z\\
&=\int_\Om (z+\wb +\tu)\cdot\nabla  z\cdot \wb +\int_\Om (z+\wb )\cdot\nabla  z\cdot \tu\\
&\leq C\nlp2{\nabla z}(\nlp4\wb ^2+\nlp4\wb \nlp4{\tu}\\
&\hskip 4cm +\nlp4\wb \nlp4z+\nlp4{\tu}\nlp4z).
\end{align*}
Next, we use the bound  $\nlp4z\leq C\nlp2z^{\frac12}\nlp2{\nabla z}^{\frac12}$ and 
we apply  the Young inequality: 
\begin{align*}
 \partial_t\nlp2z^2+\nlp2{\nabla z}^2
\leq & C\nlp4\wb ^2(\nlp4\wb ^2+\nlp4{\tu}^2)\\&+C\nlp2z^2(\nlp4\wb ^4+\nlp4{\tu}^4). 
\end{align*}
The Gronwall and H\"older inequalities imply that
\begin{equation*}
\nlp2{z(t)}^2\leq C\|\wb \|_{L^4((0,t)\times\Om)}^2\Big(\|\wb \|_{L^4((0,t)\times\Om)}^2+\|\tu\|_{L^4((0,t)\times\Om)}^2\Big)e^{C(\|\wb \|_{L^4((0,t)\times\Om)}^4+\|\tu\|_{L^4((0,t)\times\Om)}^4)}.
\end{equation*}

Since $\tu$ is a solution of the Navier-Stokes equations, the classical energy estimate reads
$$
\nlp2{\tu(t)}+2\int_0^t\nlp2{\nabla\tu}^2\leq\nlp2{\tu_0}.$$ 
Using the inequality $\nlp4{\tu}\leq C\nlp2{\tu}^{\frac12}\nlp2{\nabla \tu}^{\frac12}$, we infer that $\|\tu\|_{L^4((0,t)\times\Om)}\leq C\nlp2{\tu_0}$ for all $t>0$. 
Hence, we have
\begin{equation}\label{boundzf}
 \nlp2{z(T_\ep)}^2\leq C_1\|\wb \|_{L^4((0,T_\ep)\times\Om)}^2\Big(\|\wb \|_{L^4((0,T_\ep)\times\Om)}^2+\nlp2{\tu_0}^2\Big)
e^{C_1\left(\|\wb \|_{L^4((0,T_\ep)\times\Om)}^4+\nlp2{\tu_0}^4\right)}
\end{equation}
for some constant $C_1$.

If $\nlp2{\tu_0}$ is small, then $\|u_0\|_{L^{2,\infty}(\Om)}$ is also small and the conclusion follows after applying  Theorem \ref{corwithsmallu}. More precisely, there exists $\eta=\eta(\ep)$ such that if $\nlp2{\tu_0}\leq\eta$ then $\|\tu_0\|_{L^{2,\infty}(\Om)}\leq \frac\ep3$. Therefore if $\nlp2{\tu_0}\leq\eta$ then $\|u_0\|_{L^{2,\infty}(\Om)}\leq \ep$ and the desired conclusions follow at once from Theorem \ref{corwithsmallu}.

We assume now that $\nlp2{\tu_0}>\eta$ and that relation \eqref{wsmallbis} holds true for some constant $K$ to be determined later. We impose first that $K$ verifies the condition
$1<K\eta^2 e^{K\eta^4}  $. Then $\nlp2{\tu_0}>\eta$ and  \eqref{wsmallbis} imply that  $\|\wb\|_{L^4((0,T_\ep)\times\Om)}\leq \nlp2{\tu_0}$. We infer from \eqref{boundzf} and  \eqref{wsmallbis} that
\begin{equation*}
 \nlp2{z(T_\ep)}^2\leq 2C_1\|\wb \|_{L^4((0,T_\ep)\times\Om)}^2\nlp2{\tu_0}^2e^{2C_1\nlp2{\tu_0}^4}
\leq \frac{2C_1}{K^2}e^{2(C_1-K)\nlp2{\tu_0}^4}\leq \frac{2C_1}{K^2}
\end{equation*}
if we further assume that $K\geq C_1$. Therefore, if $K$ is chosen sufficiently large then we can conclude  that $\nlpw2{z(T_\ep)}<\frac\ep3$ implying that  
$$
\nlpw2{u(T_\ep)}\leq \nlpw2{\tu(T_\ep)}+\nlpw2{\wb(T_\ep)}+\nlpw2{z(T_\ep)}<\ep.
$$

We have reached a time where the solution $u$ becomes small in $L^{2,\infty}(\Omega)$. We wish to conclude as in the proof of the corresponding result in dimension $n\geq3$ stated in Theorem~\ref{largedata}. If we look at that proof, we notice that we need to show that $u(T_\ep)-w(T_\ep)\in L^2(\Om)$ where $w$ denotes the Navier-Stokes solution on the exterior domain with initial data $w_0$. Since $ \tu(T_\ep)\in L^2(\Om)$ and $ z(T_\ep)\in L^2(\Om)$ (as shown above) it suffices to prove that $w(T_\ep)-\wb(T_\ep)\in L^2(\Om)$. But this is a particular case of the estimates obtained above. Indeed, if we set $\tu=0$ then $u=w$ and $w(T_\ep)-\wb(T_\ep)=z(T_\ep)\in L^2(\Om)$. 

The end of the proof is similar to the end of the proof of Theorem \ref{largedata}. On one hand, we use that  $u(T_\ep)-w(T_\ep)\in L^2(\Om)$ and we apply the stability result from Theorem \ref{stability} starting from time $T_\ep$ to deduce that $u$ and $w$ have the same asymptotic behaviour as $t\to\infty$. On the other hand we use Theorem \ref{corwithsmallu} to say that $w$ and $v$ have the same asymptotic behaviour as $t\to\infty$. We conclude that $u$ and $v$ have the same asymptotic behaviour as $t\to\infty$.
\end{proof}

\begin{remark}
It is of course possible  to add a large square-integrable part to $v_0$. Assuming similar hypothesis as for the data in the exterior domain, more precisely as in relation \eqref{wsmallbis}, we can show exactly in the same manner the existence of a unique global solution which has the same asymptotic behaviour as $u$ at infinity. We do not follow this way in order to have a simpler statement for Theorem \ref{dim2}.  
\end{remark}

{\bf Acknowledgments.}
The authors would like to express their gratitude to Thierry Gallay for his helpful remarks 
which allowed them to simplify and to generalize the previous version of this work.

The work of the second author was partially supported 
by the NCN grant No.~N~N201 418839 and 
the Foundation for Polish Science operated within the
Innovative Economy Operational Programme 2007-2013 funded by European
Regional Development Fund (Ph.D. Programme: Mathematical
Methods in Natural Sciences). The third author is partially supported by the Project ``Instabilities in Hydrodynamics'' financed by Paris city hall (program ``Emergences'') and the Fondation Sciences Math\'ematiques de Paris.

\def\cprime{$'$} \def\cprime{$'$} \def\cydot{{\l}eavevmode\raise.4ex\hbox{.}}
  \def\cprime{$'$} \def\cprime{$'$}
  \def\polhk\#1{\setbox0=\hbox{\#1}{{\o}oalign{\hidewidth
  {\l}ower1.5ex\hbox{`}\hidewidth\crcr\unhbox0}}}

\adrese


\begin{thebibliography}{10}

\bibitem{amann_nonhomogeneous_2002}
H.~Amann.
\newblock Nonhomogeneous {N}avier-{S}tokes equations in spaces of low
  regularity.
\newblock In {\em Topics in mathematical fluid mechanics}, volume~10 of {\em
  Quad. Mat.}, pages 13--31. Dept. Math., Seconda Univ. Napoli, Caserta, 2002.

\bibitem{MR2143356}
P.~Biler, M.~Cannone, and G.~Karch.
\newblock Asymptotic stability of {N}avier-{S}tokes flow past an obstacle.
\newblock In {\em Nonlocal elliptic and parabolic problems}, volume~66 of {\em
  Banach Center Publ.}, pages 47--59. Polish Acad. Sci., Warsaw, 2004.

\bibitem{BBIS11}
C.~Bjorland, L.~Brandolese, D.~Iftimie, and M.~E. Schonbek.
\newblock {$L^p$}-solutions of the steady-state {N}avier-{S}tokes equations
  with rough external forces.
\newblock {\em Comm. Partial Differential Equations}, 36(2):216--246, 2011.

\bibitem{MR1158939}
W.~Borchers and T.~Miyakawa.
\newblock {$L^2$}-decay for {N}avier-{S}tokes flows in unbounded domains, with
  application to exterior stationary flows.
\newblock {\em Arch. Rational Mech. Anal.}, 118(3):273--295, 1992.

\bibitem{MR968416}
C.~P. Calder{\'o}n.
\newblock Existence of weak solutions for the {N}avier-{S}tokes equations with
  initial data in {$L^p$}.
\newblock {\em Trans. Amer. Math. Soc.}, 318(1):179--200, 1990.

\bibitem{cannone}
M.~Cannone.
\newblock Harmonic analysis tools for solving the incompressible
  {N}avier-{S}tokes equations.
\newblock In {\em Handbook of mathematical fluid dynamics. {V}ol. {III}}, pages
  161--244. North-Holland, Amsterdam, 2004.

\bibitem{Ca94}
A.~Carpio.
\newblock Asymptotic behavior for the vorticity equations in dimensions two and
  three.
\newblock {\em Comm. Partial Differential Equations}, 19(5-6):827--872, 1994.

\bibitem{cazenave_chaotic_2005}
T.~Cazenave, F.~Dickstein, and F.~B. Weissler.
\newblock Chaotic behavior of solutions of the {N}avier-{S}tokes system in
  {$\Bbb R^N$}.
\newblock {\em Adv. Differential Equations}, 10(4):361--398, 2005.

\bibitem{DS99a}
W.~Dan and Y.~Shibata.
\newblock On the {$L_q$}--{$L_r$} estimates of the {S}tokes semigroup in a
  two-dimensional exterior domain.
\newblock {\em J. Math. Soc. Japan}, 51(1):181--207, 1999.

\bibitem{DS99}
W.~Dan and Y.~Shibata.
\newblock Remark on the {$L_q$}-{$L_\infty$} estimate of the {S}tokes semigroup
  in a {$2$}-dimensional exterior domain.
\newblock {\em Pacific J. Math.}, 189(2):223--239, 1999.

\bibitem{depauw_solutions_2001}
N.~Depauw.
\newblock Solutions des \'equations de {N}avier-{S}tokes incompressibles dans
  un domaine ext\'erieur.
\newblock {\em Rev. Mat. Iberoamericana}, 17(1):21--68, 2001.

\bibitem{galdi}
G.~P. Galdi.
\newblock {\em An introduction to the mathematical theory of the
  {N}avier-{S}tokes equations. {V}ol. {I}}, volume~38 of {\em Springer Tracts
  in Natural Philosophy}.
\newblock Springer-Verlag, New York, 1994.
\newblock Linearized steady problems.

\bibitem{GM11}
T.~Gallay and Y.~Maekawa.
\newblock Long-time asymptotics for two-dimensional exterior flows with small
  circulation at infinity.
\newblock {\em To appear in Analysis \& PDE. arXiv:1202.4969}, 2012.

\bibitem{GW05}
T.~Gallay and C.~E. Wayne.
\newblock Global stability of vortex solutions of the two-dimensional
  {N}avier-{S}tokes equation.
\newblock {\em Comm. Math. Phys.}, 255(1):97--129, 2005.

\bibitem{MR635201}
Y.~Giga.
\newblock Analyticity of the semigroup generated by the {S}tokes operator in
  {$L_{r}$} spaces.
\newblock {\em Math. Z.}, 178(3):297--329, 1981.

\bibitem{GK88}
Y.~Giga and T.~Kambe.
\newblock Large time behavior of the vorticity of two-dimensional viscous flow
  and its application to vortex formation.
\newblock {\em Comm. Math. Phys.}, 117(4):549--568, 1988.

\bibitem{MR1017289}
Y.~Giga, T.~Miyakawa, and H.~Osada.
\newblock Two-dimensional {N}avier-{S}tokes flow with measures as initial
  vorticity.
\newblock {\em Arch. Rational Mech. Anal.}, 104(3):223--250, 1988.

\bibitem{GS03}
Y.~Giga and O.~Sawada.
\newblock On regularizing-decay rate estimates for solutions to the
  {N}avier-{S}tokes initial value problem.
\newblock In {\em Nonlinear analysis and applications: to {V}. {L}akshmikantham
  on his 80th birthday. {V}ol. 1, 2}, pages 549--562. Kluwer Acad. Publ.,
  Dordrecht, 2003.

\bibitem{ILK11}
D.~Iftimie, G.~Karch, and C.~Lacave.
\newblock Self-similar asymptotics of solutions to the {N}avier-{S}tokes system
  in two dimensional exterior domain.
\newblock {\em arXiv:1107.2054}, 2011.

\bibitem{kajikiya-miyakawa}
R.~Kajikiya and T.~Miyakawa.
\newblock On {$L^2$} decay of weak solutions of the {N}avier-{S}tokes equations
  in {${\bf R}^n$}.
\newblock {\em Math. Z.}, 192(1):135--148, 1986.

\bibitem{MR1689406}
G.~Karch.
\newblock Scaling in nonlinear parabolic equations.
\newblock {\em J. Math. Anal. Appl.}, 234(2):534--558, 1999.

\bibitem{kato_1984}
T.~Kato.
\newblock Strong {$L^{p}$}-solutions of the {N}avier-{S}tokes equation in
  {${\bf R}^{m}$}, with applications to weak solutions.
\newblock {\em Math. Z.}, 187(4):471--480, 1984.

\bibitem{KY95}
H.~Kozono and M.~Yamazaki.
\newblock Local and global unique solvability of the {N}avier-{S}tokes exterior
  problem with {C}auchy data in the space {$L^{n,\infty}$}.
\newblock {\em Houston J. Math.}, 21(4):755--799, 1995.

\bibitem{MR1938147}
P.~G. Lemari{\'e}-Rieusset.
\newblock {\em Recent developments in the {N}avier-{S}tokes problem}, volume
  431 of {\em Chapman \& Hall/CRC Research Notes in Mathematics}.
\newblock Chapman \& Hall/CRC, Boca Raton, FL, 2002.

\bibitem{MS97}
P.~Maremonti and V.~A. Solonnikov.
\newblock On nonstationary {S}tokes problem in exterior domains.
\newblock {\em Ann. Scuola Norm. Sup. Pisa Cl. Sci. (4)}, 24(3):395--449, 1997.

\bibitem{MR1639283}
F.~Planchon.
\newblock Asymptotic behavior of global solutions to the {N}avier-{S}tokes
  equations in {${\bf R}\sp 3$}.
\newblock {\em Rev. Mat. Iberoamericana}, 14(1):71--93, 1998.

\bibitem{MR775190}
M.~E. Schonbek.
\newblock {$L^2$} decay for weak solutions of the {N}avier-{S}tokes equations.
\newblock {\em Arch. Rational Mech. Anal.}, 88(3):209--222, 1985.

\bibitem{wiegner}
M.~Wiegner.
\newblock Decay results for weak solutions of the {N}avier-{S}tokes equations
  on {${\bf R}^n$}.
\newblock {\em J. London Math. Soc. (2)}, 35(2):303--313, 1987.

\end{thebibliography}
\end{document}